\documentclass[a4paper]{amsart}
\usepackage[utf8]{inputenc}
\usepackage[english]{babel}

\usepackage{amscd,amssymb}
\usepackage{amsmath}
\usepackage{amsthm}
\usepackage{graphicx}
\usepackage{fancybox}
\usepackage{epic,eepic}
\usepackage{amstext}
\usepackage{mathtools}
\usepackage{esint}
\usepackage[square,numbers]{natbib}
\usepackage{booktabs}
\usepackage[hide links]{hyperref}
\usepackage{comment}
\usepackage{mathtools}
\usepackage{tikz,pgfplots}
\usepackage{subcaption}

\usepackage[noabbrev, capitalise, nameinlink]{cleveref}
\crefname{equation}{}{}
\crefname{assumption}{Assumption}{Assumptions}
\crefformat{equation}{\textup{#2(#1)#3}}

\newtheorem{theorem}{Theorem}[section]

\newtheorem{lemma}[theorem]{Lemma}

\theoremstyle{definition}

\theoremstyle{remark}
\newtheorem{remark}[theorem]{Remark}
\numberwithin{theorem}{section}
\numberwithin{equation}{section}
\numberwithin{figure}{section}

\def\TH{\mathcal T_H}
\def\Th{\mathcal T_h}

\def\supp{\operatorname{supp}}

\def\with{\,:\,}
\def\dx{\,\mathrm{d}x}

\def\tint{\begingroup\textstyle \int\endgroup}
\def\Nb{\mathsf{N}}
\def\calV{\mathcal{V}}
\def\calP{\mathfrak{P}}

\DeclareMathOperator*{\argmin}{arg\,min}

\numberwithin{equation}{section}
\numberwithin{theorem}{section}

\AtBeginDocument{%
	\def\MR#1{}
}

\allowdisplaybreaks

\title[A Generalized Higher-Order LOD Framework]{A Generalized Framework for Higher-Order Localized Orthogonal Decomposition Methods}
\author[M.~Hauck, A.~Lozinski, R.~Maier]{Moritz Hauck$^*$, Alexei Lozinski$^\dagger$, Roland~Maier$^*$}
\address{${}^*$ Institute for Applied and Numerical Mathematics, Karlsruhe Institute of Technology, Englerstr.~2, 76131 Karlsruhe, Germany}
\email{\{moritz.hauck,roland.maier\}@kit.edu}
\address{${}^{\dagger}$ Université Marie et Louis Pasteur, CNRS, LmB (UMR 6623), F-25000 Besançon, France}
\email{alexei.lozinski@univ-fcomte.fr}

\begin{document}

\begin{abstract}
We introduce a generalized framework for studying higher-order versions of the multiscale method known as Localized Orthogonal Decomposition. Through a suitable reformulation, we are able to accommodate both conforming and nonconforming constraints in the construction process. In particular, we offer a new perspective on localization strategies. We fully analyze the strategy for linear elliptic problems and discuss extensions to the Helmholtz equation and the Gross--Pitaevskii eigenvalue problem. Numerical examples are presented that particularly provide valuable comparisons between conforming and nonconforming constraints.
\end{abstract}



\maketitle

\section{Introduction}

The numerical solution of partial differential equations (PDEs) with strongly heterogeneous and highly oscillatory coefficients is challenging. Classical finite element methods often fail to produce accurate results unless the fine-scale variations of the coefficients are fully resolved. 
However, globally resolving such fine-scale features can be prohibitively expensive. The associated computational cost and memory requirements often exceed the capabilities of available computing resources.
To address this, numerous multiscale methods have been developed, particularly in the elliptic setting. These include the Heterogeneous Multiscale Method~\cite{EE03,EE05,AbdEEV12}, (Generalized) Multiscale Finite Element Methods~\cite{BabO83,BabCO94,HowW97,BabL11,EfeGH13}, Multiscale Spectral Generalized Finite Element Methods~\cite{BabL11,MaSD22}, rough polyharmonic splines~\cite{OwhZB14}, the Localized Orthogonal Decomposition (LOD)~\cite{MalP14,HenP13}, and gamblets~\cite{Owh17}. More recently, refined localization strategies within the LOD framework have been proposed; see, for instance, \cite{HauP23,FreHKP24}. Comprehensive overviews of numerical multiscale methods can be found in the textbooks~\cite{OwhS19,MalP20} and the review article~\cite{AltHP21}.

While the techniques discussed above typically yield first-order convergence behavior, also higher-order multiscale methods have been developed. In the context of the Heterogeneous Multiscale Method, such extensions have been proposed in~\cite{LiPT12,AbdB12}. For the Multiscale Finite Element Method, higher-order convergence rates are achieved in~\cite{AllB05,HesZZ14}. Hybrid multiscale methods, which reduce global degrees of freedom to element boundaries, have also become quite popular to achieve higher-order rates; see, e.g.,~\cite{HarPV13,AraHPV13,CicEL19}. However, all these higher-order approaches require certain smoothness assumption on the domain, the coefficient, and/or the exact solution to obtain convergence rates beyond first order. For rough coefficients, as they arise in many applications (e.g., in the context of composite materials, that do not have a smooth transition between materials properties), such conditions are typically not fulfilled and coefficients are only in~$L^\infty$. Deriving higher-order rates in such a setting is not straightforward and requires the tailored design of approximation spaces and the use of suitable orthogonality properties. This is achieved in~\cite{Maier2021} (see also~\cite{Mai20}) with a higher-order multiscale method based on ideas of the LOD and gamblets. The approach uses higher-order non-conforming spaces as constraints in the construction of an ideal multiscale space.  It can also be reformulated within the framework of the original LOD method, as demonstrated in~\cite{Dong2023}, resulting in an improved localization strategy. The use of such constraints (also referred to as \emph{quantities of interest}) in the construction of problem-adapted approximation spaces is not new; similar ideas have been employed in the context of gamblets in~\cite{OwhS19}.

The aim of this work is two-fold. First, we reformulate the higher-order method in~\cite{Maier2021,Dong2023} within a general and unified framework, offering a new perspective on localization strategies. A key advantage of this framework is that it eliminates the need to construct quasi-interpolation operators tailored to a given set of constraints, as required in~\cite{Dong2023} to achieve improved localization. This task can be technically challenging, particularly for higher-order constraints, and often relies on intricate bubble function constructions. Moreover, the framework naturally accommodates the use of both conforming and nonconforming higher-order finite element spaces as constraints.
Second, we compare different variants of higher-order LOD-type methods and illustrate how the underlying principles can be extended to a wider class of PDEs, including heterogeneous Helmholtz problems and the Gross–Pitaevskii eigenvalue problem.

The remaining parts of the paper are organized as follows. In \cref{sec:modelproblem}, we introduce the elliptic model problem for which we define a prototypical multiscale method in \cref{sec:proto} that is able to achieve higher-order convergence rates under minimal structural assumptions. We then establish the exponential decay of the basis functions in \cref{sec:decay} and introduce a practical multiscale method using localized versions of the basis functions in \cref{sec:localization}. Finally, we discuss the applicability of the approach to the heterogeneous Helmholtz equation and the Gross--Pitaevskii eigenvalue problem in \cref{sec:general}, and present numerical examples in \cref{sec:numerics}.

\subsection*{Notation} 
Throughout this work, we will write $a \lesssim b$ or $b\gtrsim a$ if it holds that $a \leq C b$ or $a \geq C b$, respectively, where $C>0$ is a constant that is independent of the mesh size $H$, the oversampling parameter $\ell$, and oscillations of the PDE solution~$u$, but can depend on the dimension $d$, the bounds $\alpha$ and $\beta$ on the coefficient, and the polynomial degree $p$. In particular, we do not explicitly track dependencies on~$p$, as we are not considering an asymptotic behavior with respect to $p$.

\section{Model problem}
\label{sec:modelproblem}

We consider the prototypical second-order elliptic PDE \( -\mathrm{div} (A \nabla u) = f \) in weak form, with homogeneous Dirichlet boundary conditions, posed on a polygonal Lipschitz domain \( \Omega \subset \mathbb{R}^d \), where \( d \in \{2, 3\} \). The matrix-valued coefficient \( A \in L^\infty(\Omega, \mathbb{R}^{d \times d}) \) is symmetric and positive definite, satisfying for almost all $x \in \Omega$:
\begin{equation}\label{eq:propA}
	\alpha |\eta|^2 \leq (A(x)\eta) \cdot \eta \leq \beta |\eta|^2,\quad \forall \eta \in \mathbb R^d
\end{equation}
with some  constants \( 0 < \alpha \leq \beta < \infty \), where \( |\cdot| \) denotes the Euclidean norm in~\( \mathbb{R}^d \). It is important to note that no regularity assumptions are imposed on the coefficient~\( A \). The scenario of particular interest here is when the coefficient~\( A \) is rough, with oscillations across multiple length scales.

The solution space of the considered PDE is the Sobolev space \( V \coloneqq H^1_0(\Omega) \), and the associated bilinear form \( a \colon V \times V \to \mathbb{R} \) is given by
\begin{equation*}
a(u, v) \coloneqq \int_\Omega (A \nabla u) \cdot \nabla v \, \dx.
\end{equation*}
The symmetry of $A$ and the ellipticity condition \cref{eq:propA} ensure that the bilinear form~$a$ defines an inner product on the space $V$. The induced norm, referred to as the energy norm, is defined by \( \| \cdot \|_a^2 \coloneqq a(\cdot, \cdot) \), and it is equivalent to the standard \( H^1(\Omega) \)-norm on \( V \).
Given a source term \( f \in L^2(\Omega) \), the weak formulation of the problem seeks a solution \( u \in V \) such that
\begin{equation}
	\label{eq:weakform}
	a(u, v) = (f, v)_\Omega,\quad \forall v \in V,
\end{equation}
where \( (\cdot, \cdot)_\Omega \) denotes the \( L^2(\Omega) \)-inner product.
Denoting by \( \| \cdot \|_\Omega^2 \coloneqq (\cdot, \cdot)_\Omega \) the \( L^2(\Omega) \)-norm, the Riesz representation theorem ensures the existence and uniqueness of the solution \( u \in V \), along with the stability estimate
\(
	\| \nabla u \|_{\Omega} \leq \alpha^{-1} C_{\mathrm{F}} \| f \|_\Omega,
\)
where \( C_{\mathrm{F}} > 0 \) is the constant in the Friedrichs inequality on the domain \( \Omega \).

\section{Prototypical multiscale method}\label{sec:proto}

As mentioned above, classical finite element methods often perform poorly for multiscale problems: they require resolving all microscopic details of the coefficient globally, which can be prohibitively expensive, and they typically suffer from slow convergence due to the low regularity of the underlying multiscale solution.
In this section, we introduce a prototypical multiscale method that, under minimal structural assumptions on the coefficient~$A$, yields accurate (optimal-order) approximations even on very coarse meshes.
Following the standard construction of the LOD (cf.~\cite{MalP14,MalP20}), we begin by introducing a subspace of \( V \) consisting of fine-scale functions with oscillations on length scales smaller than a prescribed parameter \( H \), characterized by certain vanishing (possibly higher-order) weighted averages. 
To formalize this, we introduce a hierarchy of quasi-uniform, shape-regular, and geometrically conforming meshes \( \{ \mathcal{T}_H \}_H \), where each mesh \( \mathcal{T}_H \) is a finite subdivision of the closure of \( \Omega \) into closed simplicial\footnote{\label{page:footnote}This restriction to simplicial elements is not essential and is made for simplicity of presentation. In fact, quadrilateral/hexahedral elements may also be considered. The corresponding globally continuous finite element spaces can be constructed using polynomials on a reference element, which are then mapped to the mesh elements via multi-linear coordinate transformations.
} elements \( T \). 
The mesh size parameter \( H > 0 \) is defined as the maximum diameter of the elements in the mesh \( \mathcal{T}_H \), i.e., \( H \coloneqq \max_{T \in \mathcal{T}_H} \operatorname{diam}(T) \). In what follows, we are primarily interested in the under-resolved regime where meshes are too coarse to capture the microscopic details of the coefficient.
Given an arbitrary but fixed polynomial degree \( p \in \mathbb{N} \), we denote by \( M_H \) the $p$-th order finite element space with respect to $\TH$, associated with either the standard discontinuous Galerkin (DG) method or its continuous Galerkin (CG) counterpart. Specifically, in the DG case, \( M_H \) is given by \( \mathcal{P}^p(\mathcal{T}_H) \), the space of \( \mathcal{T}_H \)-piecewise polynomials of total degree at most \( p \), while in the CG case, the space is defined as \( \mathcal{P}^p(\mathcal{T}_H) \cap H^1(\Omega) \).
We define the fine-scale space of the LOD as
\begin{equation}
	\label{eq:defW}
	W \coloneqq \{w \in V \with (\mu,w)_\Omega=0,\;  \forall \mu\in M_H \}.
\end{equation}
The problem-adapted approximation space of the LOD is then given by the orthogonal complement of \( W \) with respect to the energy inner product \( a \), i.e., 
\begin{equation}
	\label{eq:Zms}
	\tilde{V}_H \coloneqq \big\{ v \in V \with a(v, w) = 0, \; \forall w \in W \big\}.
\end{equation}
Since \( W \) has finite codimension in \( V \), the space \( \tilde{V}_H \) is finite-dimensional.
The use of tildes in the notation highlights that these spaces are specifically tailored to the problem at hand. 
The prototypical LOD method is defined as the Galerkin projection onto \( \tilde{V}_H \), i.e., it seeks \( \tilde{u}_H \in \tilde{V}_H \) such that 
\begin{align}\label{eq:protmethod}
	a(\tilde{u}_H, \tilde{v}_H) = (f, \tilde{v}_H)_\Omega,\quad \forall \tilde{v}_H \in \tilde{V}_H.
\end{align}
Recall that \( M_H \) can be either a CG or a DG space, which leads to two distinct definitions of \( \tilde{V}_H \) and, consequently, two different (prototypical) LOD methods. We refer to these as the prototypical CG-LOD and DG-LOD methods, respectively.
Before stating a convergence result for these methods, we introduce some notation for broken Sobolev spaces. For \( s \in \mathbb{N} \), we denote by \( H^s(\TH) \) the space of \( \TH \)-piecewise \( H^s \)-functions, with the corresponding broken seminorm $|f|_{s,\TH}^2 \coloneqq  \sum_{T \in \TH} |f|_{s,T}^2$, where \( |\cdot|_{s,T} \) denotes the seminorm of order $s$ of a function in $H^s(T)$. 
\begin{theorem}[Errors of the prototypical method]\label{thm:convergenceprot}
The prototypical method \cref{eq:protmethod} is well-posed and its solution is given by \( \tilde{u}_H = \mathcal{R} u \), where \( \mathcal{R} \colon V \to \tilde{V}_H \) denotes the \( a \)-orthogonal projection onto \( \tilde{V}_H \). Moreover, for any \( f \in H^{s}(\Omega) \) for the CG-LOD, and \( f \in H^s(\mathcal{T}_H) \) for the DG-LOD, with \( s \in \{0, 1, \dots, p+1\} \), we have 
\begin{equation}
	\|u - \tilde u_H\|_\Omega \lesssim H^{s+2} |f|_{s,\TH}, \qquad\quad
	\|\nabla (u - \tilde u_H)\|_\Omega \lesssim H^{s+1} |f|_{s,\TH}\label{eq:errestL2H1}.
\end{equation}
\end{theorem}
\begin{proof}
First, note that the space \( \tilde V_H \) is a closed subspace of \( V \), and hence, problem~\cref{eq:Zms} is well-posed by the Riesz representation theorem. Comparing \eqref{eq:Zms} with the weak formulation \eqref{eq:weakform}, we observe that \( a(u - \tilde u_H, \tilde v_H) = 0 \) for all \( \tilde v_H \in \tilde V_H \), which implies that \( \tilde u_H = \mathcal R u \), where $\mathcal R\colon V \to \tilde V_H$ is the \( a \)-orthogonal projection.

Since the prototypical DG-LOD is equivalent to the prototypical method presented in~\cite{Maier2021}, the error estimates in~\eqref{eq:errestL2H1} can be directly concluded from \cite[Thm.~3.1]{Maier2021}. 
The corresponding proof for the prototypical CG-LOD follows analogously, replacing the \(L^2\)-projection onto $\TH$-piecewise polynomials in the proof of~\cite[Thm.~3.1]{Maier2021} with the \(L^2\)-projection onto globally continuous \(\mathcal{T}_H\)-piecewise polynomials, noting that both share the same approximation properties.
\end{proof}

\begin{remark}[Boundary conditions of $M_H$]\label{remBCorder} 
In the CG-LOD variant, not enforcing boundary conditions for $M_H$ is crucial to achieve higher-order convergence rates as stated in \cref{thm:convergenceprot}. For example, using the polynomial degree $p = 1$, \cref{thm:convergenceprot} gives an $H^1$-error of order $\mathcal{O}(H^3)$. In contrast, the more traditional approaches in~\cite{MalP14, HenP13, MalP20}, where the homogeneous Dirichlet boundary conditions are imposed on the analogue of $M_H$, yield an error which is only of order $\mathcal{O}(H)$.
If $f \in H^1_0(\Omega)$, convergence rates beyond first order can also be achieved using $H^1_0(\Omega)$-conforming spaces~$M_H$; see, e.g.,~\cite[Lem.~8.1]{MalP20}. This has been exploited, for instance, for the Gross--Pitaevskii problem in~\cite{HenP23}.
\end{remark}

To better understand the structure of the space \( \tilde{V}_H \), we characterize the \( a \)-orthogonal projection \( \mathcal{R} \colon V \to \tilde{V}_H \) via a saddle point formulation using Lagrange multipliers in \( M_H \). Specifically, for any \( v \in V \), the projection \( \mathcal{R}v \in V \) and the associated Lagrange multiplier \( \lambda \in M_H \) are defined as the unique solution pair to
\begin{subequations}
	\label{eq:defR} 
	\begin{align}
		&\qquad \qquad  \qquad a (\mathcal R v, w)& +&  &b(w,\lambda) & &=\quad  &0,&&\forall w \in V, \quad &&\label{eq:defR1}\\
		&\qquad \qquad  \qquad b(\mathcal R v,\mu)                   &   &         &    & &=\quad &b(v,\mu),&&\forall \mu \in M_H, \quad\label{eq:defR2}&&
	\end{align}
\end{subequations}
where the bilinear form \( b \) is given by
\(
b(v, \mu) \coloneqq (v, \mu)_\Omega.
\)
Indeed, equation~\cref{eq:defR1} ensures that \( \mathcal{R}v \in \tilde{V}_H \), while equation \cref{eq:defR2} implies that \( \mathcal{R}v - v \in W \). The well-posedness of \eqref{eq:defR} then follows from classical saddle point theory (cf.~\cite[Cor.~4.2.1]{BofBF13}) and, in particular, relies on the inf--sup condition
\begin{equation}  \label{infsupbglob} 
  \adjustlimits\inf_{\mu \in M_H} \sup_{v \in V} \frac{b (v, \mu)}{\|
  \nabla v \|_{\Omega} \| \mu \|_{\Omega}} \gtrsim  H > 0\,,
\end{equation}
which is a direct consequence of~\cite[Lem.~3.4]{Maier2021}. 

The operator \( \mathcal{R} \) and its representation given in \cref{eq:defR}, enables the construction of a basis for the space \( \tilde{V}_H = \mathcal{R}V \). Before formalizing this idea, we briefly review two existing LOD constructions from the literature:

\begin{enumerate}
\item\label{itemCorrVH} 
In~\cite{MalP14, MalP20}, the fine-scale space \( W \) is defined as the kernel of a quasi-interpolation operator \( \mathcal{I}_H \colon V \to V_H \), where \( V_H \) is an underlying \( H^1_0(\Omega) \)-conforming finite element space. A basis for \( \tilde{V}_H \) is obtained by correcting the basis functions of \( V_H \) with corrections in~\( W \).
Specifically, the prototypical LOD space~\( \tilde{V}_H \) is redefined as \( \tilde{V}_H = V_H - \mathcal{C}V_H \) with the $a$-orthogonal projection $\mathcal C \colon V \to W$ serving as a correction operator. 
This strategy is beneficial for localizing the space~\( \tilde{V}_H \) by expressing $\mathcal{C}$ as the sum of contributions from every element  of the mesh. However, it is not applicable in the present work, as our underlying finite element spaces \( M_H \) are not \( H^1_0(\Omega) \)-conforming, which is essential to obtain higher-order convergence, as discussed in Remark \ref{remBCorder}. 
\item 
In~\cite{AltHP21}, a more general framework is introduced in which the basis functions of the prototypical LOD space are associated with so-called quantities of interest (QOI), a concept that also plays an important role in the theory of gamblets; see, e.g.,~\cite{OwhS19}.
Such QOI are continuous functionals on \( V \) encoding information about  the exact solution \( u \) to problem~\cref{eq:weakform} that the method aims to preserve exactly. The fine-scale space \( W \) is then defined as the intersection of the kernels of all QOI. We adopt this general approach in the present work, choosing the QOI as the \( L^2(\Omega) \)-inner products with basis functions of \( M_H \); see~\cref{defQOI} below. 
Our DG-LOD corresponds to the choice of QOI  in~\cite[Ex.~3.1 (a) \& (b)]{AltHP21}.
The CG-LOD with \( p = 1 \) is closely related to~\cite[Ex.~3.1~(d)]{AltHP21} except for the treatment of boundary conditions of the space $M_H$, cf. \cref{remBCorder}.
Further note that the variant from \cite[Ex.~3.1~(c)]{AltHP21}  is equivalent to  the classical LOD as discussed in \cite{MalP20}. 
Modifying the boundary treatment, as in \cref{remBCorder} for the CG-LOD, yields a version of this method with an error estimate improved by one order in \( H \) compared to~\cite{MalP20}. It fits into the framework of this article by choosing \( M_H \) as a suitable subspace of \( \mathcal{P}^1(\mathcal{T}_H) \).
\end{enumerate}

In our setting, the basis functions of \( \tilde{V}_H \) are associated with QOI defined as
\begin{equation}\label{defQOI}
	q_j(v) \coloneqq \int_{\omega_j} v \Lambda_j \, \mathrm{d}x,
	\quad j = 1, \dots, J,
\end{equation}
where \( \Lambda_j \) are the basis functions of \( M_H \), \( J \coloneqq \dim(M_H) \), and \( \omega_j = \supp(\Lambda_j) \). 
In case of the CG-LOD, we choose \( \{\Lambda_j\}_{j = 1}^J\) as the $p$-th order Lagrange nodal basis\footnote{For practical reasons, we replace vertex-based functions with piecewise affine nodal functions. This modification is not essential to the construction but facilitates a simplified practical computation of the resulting basis functions in \cref{sec:locmethod}.
} of~\( M_H \). For the DG-LOD, the basis functions are taken as the \( L^2(T) \)-orthonormal Legendre bases of \( \mathcal{P}^p(T) \) on each element \( T \in \mathcal{T}_H \).
Henceforth, we denote by \( \{ \mu_j \}_{j=1}^J \) the coefficients of a function \( \mu \in M_H \) with respect to the basis \( \{ \Lambda_j \}_{j=1}^J \), i.e., 
\begin{equation}
	\label{eq:coordinates}
	\mu = \sum_{j=1}^J \mu_j \Lambda_j.
\end{equation}
The following lemma characterizes the basis functions of \( \tilde{V}_H \) as solutions to saddle point problems with a Kronecker-delta constraint on the QOI. 

\begin{lemma}[Prototypical basis]\label{le:protbasis}
	A basis of the space $\tilde V_H$ is given by $\big\{\tilde \varphi_j\with j =1,\dots,J \big\}$
	with $\tilde \varphi_j$ defined for each $j \in \{1,\dots,J\}$ as the unique solution to the saddle point problem that seeks $(\tilde \varphi_j,\lambda_j) \in V\times M_H$ such that
	\begin{subequations}
		\label{pbphiE} 
		\begin{align}
			&\qquad \qquad  \qquad a (\tilde \varphi_j, v)& +&  &b(v,\lambda) & &=\quad  &0, &&\forall v \in V,\qquad \qquad\quad &&\label{eq:LODbasis1}\\
			&\qquad \qquad \qquad b(\tilde \varphi_j,\mu)                   &   &         &    & &=\quad  &\mu_{j},&&\forall \mu \in M_H,\qquad \qquad\quad\label{eq:LODbasis3}&&
		\end{align}
	\end{subequations}
where $\mu_j$ is the coordinate of $\Lambda_j$ in the basis representation of $\mu$, cf.~\cref{eq:coordinates}.
Furthermore, the projection operator $\mathcal R$ can be represented as
    \begin{equation}
	\label{Rsumbas}
	\mathcal R v = \sum_{j = 1}^J q_j(v)\tilde \varphi_j.
\end{equation}
\end{lemma}

\begin{proof}
The well-posedness of problem \cref{pbphiE} follows from classical saddle point theory, using the inf--sup condition~\cref{infsupbglob}; see, for instance, \cite[Cor.~4.2.1]{BofBF13}. Taking $v\in W$ as test function in~\eqref{eq:LODbasis1} proves that $\tilde{\varphi}_j\in \tilde{V}_H$. To show~\cref{Rsumbas}, we define \( e \coloneqq \mathcal{R}v - \sum_{j=1}^J q_j(v) \tilde{\varphi}_j \) and observe that \( e \in W \). 
Indeed, using \cref{eq:defR2,eq:LODbasis3} and the identity   $b(v,\mu) = \sum_{j=1}^J \mu_jq_j(v)$, cf. \cref{eq:coordinates}, we find that \( q_i(e) = q_i(\mathcal{R}v) - q_i(v) = 0 \) for all \( i \in \{1, \dots, J\} \). 
Moreover, since \( e \in \tilde{V}_H \), it follows that \( e = 0 \), proving~\cref{Rsumbas} and that the functions in~\cref{pbphiE} form a basis of \( \tilde{V}_H \).
\end{proof}
	
\section{Exponential decay}\label{sec:decay}

We emphasize that the prototypical LOD basis functions defined in \cref{pbphiE} are globally supported. Therefore, their computation would require the solution to global problems, which is infeasible for practical purposes. In this section, we show that the prototypical LOD basis functions decay exponentially, which motivates their approximation by locally computable counterparts in \cref{sec:localization}. A practical multiscale method based on such local approximations is presented in \cref{sec:locmethod}.
To quantify the decay of the basis functions, and more generally of the solutions to saddle points problem of type (\ref{pbphiE}), we introduce the notion of patches with respect to the mesh $\TH$. Given an \emph{oversampling parameter} $\ell \in \mathbb N$, we define the $\ell$-th order patch of a union of elements $S \subset \TH$ recursively for $\ell \geq 2$ by $\Nb^\ell(S) \coloneqq  \Nb^1(\Nb^{\ell-1}(S))$, where $\Nb^1(S) \coloneqq \Nb(S)$ is the set of mesh elements sharing at least a node with the elements in $S$.
Before proving the exponential decay result, we state a technical lemma used in the proof, which is a reformulation of \cite[Cor.~3.6]{Maier2021}.
\begin{lemma}[Non-standard inverse inequality]\label{lem:invineq}
	There exists a constant \( C_{\mathrm{i}} > 0 \), independent of \( H \), such that for all \( T \in \mathcal{T}_H \) and all \( \mu \in M_H \)
	\begin{equation*}
		\| \mu \|_{T} \lesssim H^{-1}\| \mu \|_{- 1, T},
	\end{equation*}
	where the \( H^{-1} \)-norm is defined by
	$
	\| \mu \|_{- 1, T} \coloneqq \sup_{w \in H^1_0 (T), w \not=0 } {(\mu, w)_T}\big/{\| \nabla w \|_{T}}.
	$
\end{lemma}  

The exponential decay will be proven in the following theorem in a rather general setting. Note that the exponential decay of the prototypical LOD basis functions can be recovered by setting \( f_S(v) = 0 \) for all $v \in V$ and \( g_S(\mu) = \mu_j \) for all \( \mu \in M_H \), where \( \mu_j \) is the coordinate of \( \Lambda_j \) in the basis expansion of \( \mu \), cf.~\cref{eq:coordinates}.

\begin{theorem}[Exponential decay]\label{thm:dec}
    Let $S$ be a union of mesh elements $S \subset \TH$ and $\psi\in V$ be the solution to the problem: find $(\psi,\lambda)\in V\times M_H$ such that
    \begin{subequations}
		\label{pbpsi} 
		\begin{align}
			&\qquad\qquad \qquad a (\psi, v)& +&  &b(v,\lambda) & &=\quad  &f_S(v),&&\forall v \in V,\qquad \qquad\quad &&\label{eq:psi1}\\
			&\qquad\qquad\qquad b(\psi,\mu)                   &   &         &    & &=\quad  &g_S(\mu),&&\forall \mu \in M_H,\qquad \qquad\quad\label{eq:psi3}&&
		\end{align}
	\end{subequations}
	where $f_S$ and	$g_S$ are bounded linear functionals on  $V$ and $M_H$, respectively, such that $f_S(v) = 0$ for all $v \in V$ with$ \supp(v) \subset \overline{\Omega \setminus S}$ and analogously for~$g_S$. 
Then there exists \( C_\mathrm{dec}> 0 \), independent of \( H \), \( \ell \), and \( S \), such that for all \( \ell \in \mathbb{N} \) 
\begin{equation}
	\label{eq:dec}
	\|\nabla \psi\|_{\Omega \setminus \mathsf{N}^\ell(S)} \leq \exp(-C_\mathrm{dec} \ell) \|\nabla \psi\|_\Omega.
\end{equation}
\end{theorem}

\begin{proof} 
The proof is based on ideas from~\cite{MalP14,HenP13,MalP20}. For the DG-LOD, it follows directly from the arguments in~\cite[Thm.~4.1]{Maier2021}. In the CG-LOD case, some modifications are necessary, as outlined below.
Fix an integer \( \ell \ge 1 \), and let \( \eta \in W^{1,\infty}(\Omega) \) be the first-order finite element cut-off function characterized by
  \begin{equation}
      \eta  =0 \text{ in } \Nb^{\ell - 1} (S),\quad
      \eta  =1 \text{ in } \Omega \setminus \Nb^{\ell} (S),  \label{eq:eta}
  \end{equation}
  with the transition region $R \coloneqq \Nb^\ell(S)\setminus \Nb^{\ell-1}(S)$, satisfying
  \begin{equation}
    \label{eq:boundeta} 
    \| \eta \|_{L^\infty(\Omega)} \le 1, \quad
    \| \nabla \eta \|_{L^\infty(\Omega)} \lesssim H^{-1}.
  \end{equation} 
Furthermore, let
\begin{equation}\label{restricta}
	 a_\omega(v, w) \coloneqq \int_\omega (A \nabla v) \cdot \nabla w \, \mathrm{d}x
\end{equation}
denote the restriction of the bilinear form \( a \) to a subdomain \( \omega \subset \Omega \). 
Testing \cref{eq:psi1} with \( \eta \psi \) and decomposing the suport of $\eta $ into $\Omega \setminus \mathsf{N}^\ell(S)$ and $R$, we obtain that
\[
a_{\Omega \setminus \mathsf{N}^\ell(S)}(\psi, \eta \psi) = f_S(\eta \psi) - a_R(\psi, \eta \psi) - b(\eta \psi, \lambda).
\]
Using the lower bound from~\eqref{eq:propA}, we estimate this by
  \begin{align} \label{eq:decaymanipulations}
    \alpha \| \nabla \psi \|_{\Omega \setminus \Nb^{\ell} (S)}^2  \leq |f_S(\eta \psi)| + |a_R (\psi, \eta \psi)| + |b(\eta \psi, \lambda)|  
     \eqqcolon \Xi_1 + \Xi_2 + \Xi_3. 
  \end{align}

  We now consider terms \( \Xi_1\)--\(\Xi_3 \) individually. We have \( \Xi_1 = 0 \) since \( (\eta \psi)|_S =0 \). To estimate \( \Xi_2 \), we apply \cref{eq:propA,eq:boundeta}, which yields
  \begin{equation}\label{eq:boundXi2} 
     \Xi_2 \leq \beta \| \nabla \psi \|_R  \| \nabla (\eta \psi)\|_R \lesssim \| \nabla \psi \|_R^2 + H^{- 1} \| \nabla \psi \|_R \|
     \psi \|_R.
  \end{equation}
 To estimate the right-hand side of the previous inequality, let \( \Lambda_z^1 \) denote the \( \mathcal{P}^1 \)-finite element hat function associated with the node \( z \), and define \( \omega_z \coloneqq \supp(\Lambda_z^1) \). Then, we have the Poincaré-type inequality
 \begin{equation}
 	\label{eq:pctype}
   	\| v \|_{\omega_z} \leq C_\mathrm{P} H \| \nabla v \|_{\omega_z}
 \quad \forall v \in \big\{H^1(\omega_z) \with \tint_{\omega_z} \Lambda_z^1 v \dx = 0\big\},
 \end{equation}
 which can be shown with the Peetre--Tartar lemma (see, e.g.,  \cite[Lem.~A.38]{ErG04}), a scaling argument, and the maximization of $C_\mathrm{P}$  over all the patch configurations that are admissible by the mesh regularity.
To be able to apply \cref{eq:pctype} for estimating the right-hand side of \cref{eq:boundXi2}, we cover \( R \) by a collection of patches \( \omega_z \) such that the union of the $\omega_z$ equals 
\( R^2 \coloneqq \operatorname{int}\big( \mathsf{N}^{\ell+1}(S) \setminus \mathsf{N}^{\ell - 1}(S) \big).\) 
Testing \cref{eq:psi3} with~\( \Lambda_z^1 \) and noting that \( \omega_z \cap S = \emptyset \) gives  $
\int_{\omega_z} \Lambda_z^1 \psi \dx = 0$.
Thus, applying~\cref{eq:pctype} to~\( \psi \) locally on each patch in the collection and summing up, we obtain that  
\(
\| \psi \|_{R} \lesssim H \| \nabla \psi \|_{R^2},
\)
using the finite  overlap of the  \( \omega_z \).
Substituting this into \cref{eq:boundXi2}, we get
  \begin{equation}
    \label{eq:Xi2} \Xi_2  \lesssim  \| \nabla \psi \|_{R^2}^2.
  \end{equation}
  
  To estimate \( \Xi_3 \), we localize it to the ring \( R^2 \). To this end, we decompose~\( \lambda \) as \( \lambda = \lambda^\mathrm{in} + \lambda^\mathrm{out} \), where \( \lambda^\mathrm{in} \) is the linear combination of basis functions of \( M_H \) supported only in \( \mathsf{N}^{\ell+1}(S) \), and \( \lambda^\mathrm{out} \) is the combination of the remaining basis functions, associated with Lagrange points not in the interior of \( \mathsf{N}^{\ell+1}(S) \). This decomposition is  well-defined and unique. Since \( \eta \psi = \psi \) in  \( \Omega \setminus \mathsf{N}^{\ell}(S) \) and \( \lambda^\mathrm{out} = 0 \) in \( \mathsf{N}^{\ell}(S) \), it follows from \eqref{eq:psi3} that \( b(\eta \psi, \lambda^\mathrm{out}) = b(\psi, \lambda^\mathrm{out}) = 0 \). Thus,
\begin{equation}
	\label{eq:x13firstest}
	\Xi_3 = 
    b  (\eta \psi, \lambda^\mathrm{in}) =
	\int_{R^2} \eta \psi \lambda^\mathrm{in} \dx
	\leq \| \eta \psi \|_{R^2}\| \lambda^\mathrm{in} \|_{R^2} 
	\lesssim \| \psi \|_{R^2}\| \lambda\|_{R^2},
\end{equation}
noting that the integral above could be restricted to \( R^2 \) since \( \eta = 0 \) on \( \mathsf{N}^{\ell - 1}(S) \) and \( \lambda^\mathrm{in} = 0 \) outside \( \mathsf{N}^{\ell + 1}(S) \). The last estimate in \cref{eq:x13firstest} follows from \( \| \lambda^\mathrm{in} \|_{R^2} \lesssim  \| \lambda \|_{R^2} \), which can be proved by a scaling argument on each element, noting that \( \lambda^\mathrm{in} \) is uniquely determined by \( \lambda \).
 To further estimate \cref{eq:x13firstest}, we consider an arbitrary element $T\subset \overline{R^2}$ and observe  $\|\lambda\|_T \lesssim H^{-1}\| \lambda \|_{- 1, T}$ by \cref{lem:invineq}. Testing~\eqref{eq:psi1} with any \( v \in H^1_0(T) \) satisfying \( \|\nabla v\|_T = 1 \)  and noting that $f_S (v) = 0$, yields 
  \begin{equation*}
  	(\lambda,v)_T   = b(v, \lambda) = - a (\psi, v) 
    \leq \beta \| \nabla \psi \|_T \| \nabla v \|_T  ,
  \end{equation*}
  which implies that $\| \lambda \|_{- 1, T} \leq \beta \| \nabla \psi \|_T$. Applying \cref{eq:pctype} as above, combining the previous estimates, and summing over all \( T \subset \overline R \), we obtain that
  \begin{equation}\label{eq:Xi3} 
    \Xi_3 \lesssim 
     H\| \nabla \psi \|_{R^2} \Big(\sum_{T \subset R^2}  H^{-2} \| \nabla \psi \|_T^2 \Big)^{1/2}
    \lesssim \| \nabla \psi \|_{R^2}^2\,.
  \end{equation}
  
Substituting~\cref{eq:Xi2,eq:Xi3} into~\eqref{eq:decaymanipulations}, and noting that the ring \( R^2 \) can be written as \( R^2 = \operatorname{int}\big( (\Omega \setminus \Nb^{\ell - 1}(S)) \setminus (\Omega \setminus \Nb^{\ell+1}(S)) \big) \), we conclude, for a constant $C>0$ independent of $H$ and $\ell$ that
  \begin{align*}
    \| \nabla \psi \|_{\Omega \setminus \Nb^{\ell} (S)}^2 \leq C \| \nabla
    \psi \|_{R^2}^2 = C\| \nabla \psi \|_{\Omega \setminus \Nb^{\ell - 1} (S)}^2 -
    C \| \nabla \psi \|_{\Omega \setminus \Nb^{\ell+1} (S)}^2,
  \end{align*} 
which, after rearranging the terms, leads to
  \begin{equation}\label{jump2lay} 
   \| \nabla \psi \|_{\Omega \setminus \Nb^{\ell+1} (S)}^2 \leq \frac{C}{1 + C}  \,\| \nabla \psi \|_{\Omega \setminus \Nb^{\ell - 1}
     (S)}^2.
   \end{equation} 
   Iterating the argument gives the assertion with $C_\mathrm{dec}\coloneqq\frac 14 \log \frac{1+C}{C}$.
\end{proof}

\begin{remark}[Better decay rate for DG-LOD]
As mentioned above, the proof of exponential decay for the DG-LOD is very similar. 
Inequality~\eqref{eq:pctype} can be replaced by the standard Poincaré inequality for mean-zero functions on each mesh element. Thus, estimate \cref{eq:Xi2} can be localized to  ring~\( R \) of width one instead of \( R^2 \), noting that piecewise constants lie in \( M_H \).
Likewise, in \eqref{eq:x13firstest} the form \( b \) localizes to~\( R \), since contributions outside \( \Nb^\ell(S) \) vanish.  Therefore,~\eqref{jump2lay} can be replaced by
\begin{equation}
	\label{eq:dec2}
	 \| \nabla \psi \|_{\Omega \setminus \Nb^{\ell} (S)}^2 \leq \frac{C}{1 + C}  \,\| \nabla \psi \|_{\Omega \setminus \Nb^{\ell - 1}(S)}^2
\end{equation}
 with a constant $C>0$, and the resulting decay rate is $\frac 12 \log \frac{1+C}{C}$. Although $C$ might not be the same as in \cref{jump2lay}, the more local nature of estimate \cref{eq:dec2} gives some theoretical underpinning of the better localization properties of the DG-LOD compared to the CG-LOD. This is confirmed by numerical experiments in \cref{sec:numerics}. 
\end{remark}

\section{Localization}
\label{sec:localization}

The exponential decay properties established in the previous section motivate the localized computation of the prototypical LOD basis functions. However, the naive strategy of localizing these basis functions by restricting problem~\eqref{pbphiE} to \( \ell \)-th order patches has the drawback that the localization error increases when decreasing the mesh size if the parameter \( \ell \) is kept fixed, cf.~\cite{MalP14,Maier2021}.
One strategy to address this, described in \cref{itemCorrVH} on \cpageref{itemCorrVH}, expresses the prototypical LOD space  \( \tilde{V}_H \) as \( \tilde{V}_H = V_H - \mathcal{C}V_H \), where \( \mathcal{C} \colon V \to W \) denotes the \( a \)-orthogonal projection onto the fine-scale space \( W \). The operator $\mathcal C$ is then decomposed into a sum of element-wise contributions, which are subsequently localized. Note that the projection operator \( \mathcal{R} \) can then be written as \( \mathcal{R} = \mathcal{I}_H - \mathcal{C}\mathcal{I}_H \) if $\mathcal{I}_H\colon V\to V_H$ is a projection. However, this strategy requires the fine-scale space to be defined as the kernel of a quasi-interpolation operator \( \mathcal{I}_H \colon V \to V_H \), where~\( V_H \) is some $H^1_0(\Omega)$-conforming finite element space. In the present setting, this construction is not directly applicable, as the space \( M_H \) is not \( H^1_0(\Omega) \)-conforming. 

Our localization approach mimics the above construction by expressing \( \mathcal{R} \) as
\begin{equation}\label{defK}
	\mathcal{R} = \mathcal{I}_H - \mathcal{K},
\end{equation}
where the operator \( \mathcal{K} \colon V \to V \) will be characterized below, and \( \mathcal{I}_H \colon V \to V_H \) is a quasi-interpolation operator onto the first-order \( H^1_0(\Omega) \)-conforming finite element space \( V_H \) associated with the mesh \( \mathcal{T}_H \). 
Note that \( V_H \) is fixed as the first-order finite element space regardless of the polynomial degree \( p \).
The quasi-interpolation operator \(\mathcal{I}_H\) (which is not necessarily a projection) is assumed to satisfy classical local approximation and stability estimates, i.e.,  
\begin{equation}
	\label{eq:propIH}
	H^{-1}\|v - \mathcal{I}_H v\|_T + \|\nabla \mathcal{I}_H v\|_T \lesssim 
	 \|\nabla v\|_{\mathsf{N}(T)},\quad \forall T \in \TH,\; \forall v \in V.
\end{equation}
Moreover, the operator \(\mathcal{I}_H\) should depend on its argument only through its QOI defined in \cref{defQOI}, i.e., \(\mathcal{I}_H w = 0\) for all $w \in W$. Such operators can be readily constructed for the two classes of methods considered. 
Indeed, for the CG-LOD, a weighted Cl\'ement-type interpolation can be used. This interpolation is uniquely defined by assigning values at interior nodes $z \in \mathcal N_H^\mathrm{in}$ as
\begin{equation}\label{defIHcg} 
	(\mathcal{I}_H v)(z) \coloneqq \left(\int_{\omega_z} \Lambda_z^1 \, \mathrm{d}x \right)^{-1} \int_{\omega_z} v \, \Lambda_z^1 \, \mathrm{d}x, \quad \forall v \in V,
\end{equation}
where \(\omega_z\) is the support of the first-order finite element hat function \(\Lambda_z^1\) associated with node \(z\). Nodal values at boundary vertices are set to zero. Such operators were, e.g., introduced in~\cite{Car99}.
For the DG-LOD method, a quasi-interpolation can be uniquely defined by specifying its nodal values at interior nodes \( z \in \mathcal{N}_H^\mathrm{in} \) as
\begin{equation}\label{defIHdg}
   	(\mathcal{I}_H v)(z) \coloneqq \sum_{T\in\omega_z}\frac{|T|}{|\omega_z|}\int_{T} v \dx,\quad \forall v \in V,
\end{equation}
where $|\cdot|$ denotes the $d$-dimensional volume of a subdomain, 
and values at boundary nodes are again set to zero; see, e.g., \cite{ErnG17} for an analysis in a more general setting.
    
The operator \( \mathcal{K} \colon V \to V \), as introduced in \cref{defK}, is characterized for each \( v \in V \) as the unique solution \( (\mathcal{K}v, \lambda) \in V \times M_H \) to the saddle point problem
\begin{subequations}
	\label{eq:defK} 
	\begin{align}
		&\qquad\quad  a (\mathcal K v, w)& +&  &b(w,\lambda) & &=\quad  &a(\mathcal I_H v,w),&&\forall w \in V,\quad \quad &&\label{eq:defK1}\\
		&\qquad\quad b(\mathcal K v,\mu)                   &   &         &    & &=\quad  &-c(v,\mu),&&\forall \mu \in M_H,\quad \quad\label{eq:defK2}&&
	\end{align}
\end{subequations}
where we use the abbreviation $c(v,\mu) \coloneqq b(v-\mathcal I_H v,\mu)$.
The operator \( \mathcal{K} \) can now be expressed as the following sum of local element contributions, i.e.,
\begin{equation*}
	\mathcal{K} = \sum_{T \in \mathcal{T}_H} \mathcal{K}_T,
\end{equation*}
where, for all $T \in \TH$, the operators \( \mathcal{K}_T \colon V \to V \) are defined for \( v \in V \) as the unique solution \( (\mathcal{K}_T v, \lambda_T) \in V \times M_H \) to the modified saddle point problem
	\begin{subequations}
	\label{eq:defKT} 
	\begin{align}
		&\qquad\quad  a (\mathcal K_T v, w)& +&  &b(w,\lambda_T) & &=\quad  &a_T(\mathcal I_H v,w),&&\forall w \in V,\quad\quad &&\label{eq:defKT1}\\
		&\qquad\quad b(\mathcal K_T v,\mu)                   &   &         &    & &=\quad  &-c_T(v,\mu),&&\forall \mu \in M_H.\quad \quad\label{eq:defKT2}&&
	\end{align}
\end{subequations}
Here, in contrast to \cref{eq:defK}, we use localized versions of the bilinear forms \( a \) and \( c \) on the right-hand side of the problem, denoted by \( a_T \) and \( c_T \), respectively. The local form~\( a_T \) is defined by restricting $a$ to $T$, as in (\ref{restricta}).
To define the local form \( c_T \), there are generally several meaningful options. We formulate the following three properties that the bilinear form \( c_T \) should satisfy:
\begin{enumerate}\label{propbtilde}
    \item \emph{Summation property:} $\sum_{T \in \TH} c_T(v,\mu) =  c(v,\mu)$ for all $v\in V$ and  $\mu\in M_H$;\label{cond1}
    \item \emph{Vanishing on constants:} $c_T (v,\mu)= 0$ for all $\mu\in M_H$ if $v =1$ on $\mathsf N(T)$ for elements $T$ that do not contain any boundary nodes;\label{cond2}
    \item \emph{Vanishing on fine scales:} $c_T(w,\mu) = 0$ for all $w \in W$ and $\mu \in M_H$;\label{cond3}
    \item \emph{Locality:} $|c_T(v,\mu)| \lesssim \|v\|_{\mathsf N(T)} \|\mu\|_{\mathsf N(T)}$ for all $v \in V$  and $\mu \in M_H$.\label{cond4}
\end{enumerate}
For the CG-LOD, the form $c_T$ can be constructed as
\begin{equation}\label{tildebT}
    c_T(v,\mu) \coloneqq \sum_{j=1}^J 
    \frac{|T\cap\omega_j|}{|\omega_j|}
    {\mu}_j q_j (v) -   \int_T \mu \mathcal{I}_H v \dx,
\end{equation}     
where \( \mu_j \) is the coordinates of $\Lambda_j$ in the basis expansion of $\mu$, cf.~\cref{eq:coordinates}.
The properties in \cref{cond1,cond3} can be verified directly, and \cref{cond4} follows from the equivalence of norms on the finite dimensional space $M_H$ restricted to $\mathsf N(T)$. The property in \cref{cond2} follows from the fact that the operator \( \mathcal{I}_H \) preserves constants on an element $T$ that does not contain any boundary nodes, and $ \int_T \Lambda_j\dx =\frac{|T|}{|\omega_j|}{\int_{\omega_j} \Lambda_j\dx }$ provided that $T\subset \overline{\omega_j}$, using the change of variables formula to relate the element integrals.
For the DG-LOD method, \eqref{tildebT} can be used as well and simplifies to
\begin{equation*}
	c_T(v, \mu) = \int_T (v - \mathcal{I}_H v)\, \mu \, \mathrm{d}x,
\end{equation*}
since the supports $\omega_j$ of any $\Lambda_j$ consists of only one mesh element in this case. 

\

Due to the general setting considered in \cref{thm:dec}, the result applies to the function \( \mathcal{K}_T v \) for any \( v \in V \), and shows that it decays exponentially away from the element \( T \). This motivates the localization of the operator \( \mathcal{K}_T \) to $\ell$-th order patches around $T$. To this end, we introduce the corresponding localized spaces
\begin{align*}
	V_T^\ell &\coloneqq \{v \in V \with \supp(v) \subset \mathsf N^\ell(T)\}, \\
	M_T^\ell & \coloneqq \{\mu|_{\Nb^\ell(T)}  \with \mu\in M_H^p\},
\end{align*}
where we implicitly extend functions in the space $M_T^\ell$ by zero.
The localized operator \( \mathcal{K}_T^\ell \colon V \to V_T^\ell \) can then be defined, for any \( v \in V \), as the unique solution \( (\mathcal{K}_T^\ell v, \lambda_T^\ell) \in V_T^\ell \times M_T^\ell \) to the local saddle point problem
	\begin{subequations}
	\label{eq:defKTell} 
	\begin{align}
		&\qquad\quad  a (\mathcal K_T^\ell v, w)& +&  &b(w,\lambda_T^\ell) & &=\quad  &a_T(\mathcal I_H v,w),&&\forall w \in V_T^\ell,\quad\quad &&\label{eq:defKTell1}\\
		&\qquad\quad b(\mathcal K_T^\ell v,\mu)                   &   &         &    & &=\quad  &-c_T(v,\mu),&&\forall \mu \in M_T^\ell.\quad \quad\label{eq:defKTell2}&&
	\end{align}
\end{subequations}
Finally, a localized version of the operator $\mathcal R$ can be defined by
\begin{equation}
	\label{eq:defRl}
	\mathcal R^\ell v\coloneqq (\mathcal I_H - \mathcal K^\ell)v,\qquad \mathcal K^\ell v \coloneqq \sum_{T \in \mathcal{T}_H} \mathcal K_T^\ell v.
\end{equation}
The following theorem shows that the localized operator \( \mathcal{K}^\ell \) approximates \( \mathcal{K} \) exponentially well in the operator norm. Notably, it avoids the \( H^{-1} \) prefactor typically arising in naive localization strategies; see, e.g.,~\cite{MalP14,Maier2021}.

\begin{theorem}[Localization error]
	\label{thm:locerr}
        For all \( v \in V \) and \( \ell \in \mathbb{N} \), we have
		\begin{equation}
			\label{eq:locerrR}
			\|\nabla (\mathcal{R} - \mathcal{R}^\ell) v\|_{\Omega} \lesssim \ell^{(d-1)/2} \exp(-C_\mathrm{dec} \ell)\|\nabla \mathcal{R} v\|_\Omega,
		\end{equation}
		where \( C_\mathrm{dec} \) is the constant from \cref{thm:dec}.
\end{theorem}
\begin{proof} We suppose without loss of generality that $\ell \ge 2$, observing that the case $\ell = 1$ holds by scaling.
We abbreviate the localization error by \( e \coloneqq (\mathcal{R} - \mathcal{R}^\ell)v \) and note that \( e \in W \). Indeed, by \cref{defK,eq:defRl,eq:defKT2,eq:defKTell2} we have $b(e,\mu)=-\sum_{T\in\TH}c_T(\mathcal{K}_Tv-\mathcal{K}_T^\ell v,\mu)=0$ for all $\mu\in M_H$. Thus,  
\begin{equation}
\begin{aligned}\label{eq:spliterror}
	\alpha\|\nabla (\mathcal R - \mathcal R^\ell)v\|_\Omega^2 
    \leq -a(\mathcal R^\ell v,e) 
    =  \sum_{T \in \mathcal{T}_H} \Big(-a_T(\mathcal I_Hv,e) + a(\mathcal K_T^\ell v,e)\Big).
\end{aligned}
\end{equation}
In the following, we consider each summand on the right-hand side separately. We use the cut-off function from \cref{eq:eta} with $S=T$, now denoted by \( \eta_T \). 
Note that \( \eta_T =0 \) on \( T \), so that $a_T(\mathcal{I}_H v, e) = a_T(\mathcal{I}_H v, (1 - \eta_T) e).$ Together with \cref{eq:defKTell1} for the test function \(w =  (1 - \eta_T)e \in V_T^\ell \), this yields
\begin{align*}
	-a_T(\mathcal I_Hv,e)+a(\mathcal K_T^\ell v,e) &= -a_T(\mathcal I_Hv,(1-\eta_T)e)+a(\mathcal K_T^\ell v,(1-\eta_T)e + \eta_T e)\\
	&= - b((1-\eta_T) e,\lambda_T^\ell) +a(\mathcal K_T^\ell v,\eta_T e)\eqqcolon \Xi_1+\Xi_2.
\end{align*}
To estimate the term $\Xi_1$,   we decompose~\( \lambda^\ell_T = \lambda^\mathrm{in}_T + \lambda^\mathrm{out}_T \), where \( \lambda^\mathrm{in}_T \) is the linear combination of basis functions of \( M_H \) fully supported in \( \mathsf{N}^{\ell-1}(T) \), and \( \lambda^\mathrm{out}_T \) is the combination of the remaining basis functions. Specifically, for the CG-LOD, the basis functions forming \( \lambda^\mathrm{out}_T \) are associated with Lagrange points that do not lie in \( \operatorname{int}(\mathsf{N}^{\ell-1}(T)) \). For DG-LOD, we simply choose \( \lambda^\mathrm{out}_T = \lambda^\ell_T|_{\Omega\setminus\Nb^{\ell-1}(T)}\). 
Since \( \eta  = 0 \) in  \( \mathsf{N}^{\ell-1}(S) \), it follows from \eqref{eq:defKTell2} that \( b((1-\eta_T)e, \lambda^\mathrm{in}) = b(e, \lambda^\mathrm{in}) = 0 \). Thus,  introducing  $R_T^2 \coloneqq \operatorname{int}\big( \Nb^{\ell}(T) \setminus \Nb^{\ell - 2}(T) \big)$, we obtain that
\begin{align*}
	\Xi_1 = -b((1-\eta_T)e,\lambda_T^\mathrm{out}) 
    \le \int_{R_T^2} |e\lambda_T^\mathrm{out}| \dx
    \lesssim \|e\|_{R_T^2} \|\lambda_T^\ell\|_{R_T^2}.
\end{align*}
To estimate the norm of $\lambda_T^\ell$, we employ  \cref{lem:invineq} on each mesh element $K\subset \overline{R_T^2}$, resulting in a negative power of $H$ and the $H^{-1}(K)$-norms of $\lambda_T^\ell$. We now take any $w\in H^1_0(K)$ with $\|\nabla w\|_{K}=1$ as a test function in \cref{eq:defKTell1} and note that $a_T(\mathcal I_H v,w) = 0$ since $T \cap R_T^2 = \emptyset$ for $\ell\ge 2$.  Thus, 
\begin{align*}
	|(\lambda_T^\ell,w)_K| = |a(\mathcal K_T^\ell v,w)| \leq \beta \|\nabla \mathcal K_T^\ell v\|_{K} ,
\end{align*}
yielding $\|\lambda_T^\ell\|_{K}\lesssim \|\lambda_T^\ell\|_{-1,K} \leq \beta H^{-1} \|\nabla \mathcal K_T^\ell v\|_{K}$. 
Recalling that $\|e\|_{R_T^2}\lesssim H \|\nabla e\|_{R_T^2}$ by the Poincaré-type inequality~\eqref{eq:pctype} (or the usual Poincaré inequality for the DG-LOD), 
and substituting this into the estimate for $\Xi_1$, we obtain that
\begin{equation}
	\Xi_1 \lesssim  \|\nabla  \mathcal K_T^\ell v\|_{R_T^2}\|\nabla e\|_{R_T^2}.
\end{equation}
Turning to $\Xi_2$, we remark that the contributions to this term also vanish on all the mesh elements outside $R_T^2$ since $\ell\geq 2$.  As above, we arrive at 
\begin{align*}
	\Xi_2 \lesssim  \|\nabla \mathcal K_T^\ell v\|_{R_T^2}(\|\nabla e\|_{R_T^2}+ \|e\|_{R_T^2})
        \lesssim \|\nabla \mathcal K_T^\ell v\|_{R_T^2}\|\nabla e\|_{R_T^2}.
\end{align*}
Putting the above estimates together, we obtain that
\begin{equation}\label{eq:bounde}
\begin{aligned}
	\|\nabla e\|_\Omega^2 &\lesssim  \sum_{T \in \TH} \|\nabla \mathcal K_T^\ell v\|_{R_T^2}\|\nabla e\|_{R_T^2} \lesssim  \exp(-C_\mathrm{dec}\ell) \sum_{T \in \TH}\|\nabla \mathcal K_T^\ell v\|_{\Nb^{\ell}(T)}\|\nabla e\|_{R_T^2},
\end{aligned}
\end{equation}
where we applied \cref{thm:dec} to \cref{eq:defKTell}, treating the patch $\mathsf N^\ell(T)$ as the whole domain.
In order to pass from the norm of $\mathcal K_T^\ell$ to that of $v$, we again use classical saddle point theory (see, e.g.~\cite[Cor.~4.2.1]{BofBF13}), 
recalling that the inf--sup constant of $b$ is of order $H$ as outlined in \cref{infsupbglob}.  This results in
\begin{equation}\label{eq:boundKTl}
\begin{aligned}
   \| \nabla \mathcal{K}_T^{\ell} v\|_{\Nb^{\ell} (T)} & \lesssim
  \sup_{w \in V_T^\ell}\frac{a_T (\mathcal{I}_H v, w)}{\| \nabla w
  \|_{\Omega}} + H^{-1} \sup_{\mu \in M_T^\ell} \frac{c_T (v,
  \mu)}{\| \mu \|_{\Omega}}\\
  & \lesssim  \| \nabla v\|_{\Nb (T)}  + H^{- 1} \|v - \bar{v}_T \|_{\Nb
  (T)} \lesssim  \| \nabla v\|_{\Nb (T)}
\end{aligned}
\end{equation}
where $\bar{v}_T$ is the average of $v$ over $\Nb (T)$ if $T$ contains a boundary node, and $\bar{v}_T=0$ otherwise. The justification of the last inequality in~\eqref{eq:boundKTl} requires to consider the cases where $T$ is completely inside $\Omega$ and $T$ is adjacent to the boundary separately. In the first case, we can subtract the average $\bar{v}_T$ inside the bilinear form $c_T$ due to \cref{cond2}
on page \pageref{propbtilde}, and then proceed by using \cref{cond4} and the Poincaré inequality on~$\Nb(T)$. Note that the constant in the Poincaré inequality is of order $H$, as seen from scaling and maximizing over all possible configurations of $\Nb(T)$ allowed by the mesh regularity. In the other case where $T$ lies at the boundary $\partial\Omega$, at least one of the boundary faces of $\Nb(T)$ lies on $\partial\Omega$ and we can conclude by the Friedrichs-type inequality $\|v\|_{\Nb  (T)} \lesssim  H \| \nabla v\|_{\Nb (T)}$,  which holds since $v$ vanishes on $\partial\Omega$. 

Returning to~\eqref{eq:bounde}, we conclude that
\begin{align*}
    \|\nabla e\|_\Omega^2&\lesssim  \exp(-C_\mathrm{dec}\ell) \sqrt{\sum_{T \in \TH} \|\nabla v\|_{\Nb(T)}^2}\sqrt{\sum_{T \in \TH} \|\nabla e\|_{R_T^2}} \\
    & \lesssim  \ell^{(d-1)/2}\exp(-C_\mathrm{dec}\ell) \|\nabla v\|_\Omega \|\nabla e\|_\Omega, 
\end{align*}
where we have used the fact that each element $K\in\TH$ belongs to at most $\mathcal O(\ell^{d-1})$ rings $R_T^2$ for different $T\in\TH$. Dividing by $\|\nabla e\|_\Omega$ gives the assertion.
\end{proof}

\section{Localized multiscale method}
\label{sec:locmethod}

In this section, we introduce a practical multiscale method with locally computable basis functions. This is reasonable due to the exponential decay behavior of the globally defined functions. We define the localized multiscale  space as $\tilde V_H^\ell=\mathcal{R}^{\ell}V$. Using that the operator \( \mathcal{R}^\ell \) depends on its argument only through the~QOI introduced in \cref{defQOI}, allows us to write 
\begin{equation}
	\label{eq:approxspace}
	\tilde V_H^\ell \coloneqq \operatorname{span}\{\tilde \varphi_j^\ell \with j = 1,\dots, J\},
\end{equation}
with appropriate basis functions associated with the QOI by the property that $q_i(\tilde{\varphi}_j^\ell)=\delta_{ij}$.
Formally, these basis functions can be written as $\tilde{\varphi}_j^\ell = \mathcal{R}^{\ell} \tilde{\varphi}_j$. This property allows us to practically compute \( \tilde{\varphi}_j^\ell \) as a linear combination of the solution to a few local problems as explained in the following.
To this end, we compute the coefficients \( \kappa_{zj} \) for $z \in \mathcal{N}_H^{\mathrm{in}}$ and $j=1,\ldots,J$ such that  $\mathcal{I}_H \tilde{\varphi}_j =\sum_{z \in \mathcal{N}_H^{\mathrm{in}}} \kappa_{zj} \Lambda_z^1$, where $\mathcal{N}_H^{\mathrm{in}}$ denotes the set of interior nodes. These coefficients are easily deduced from the definitions in~\cref{defIHcg,defIHdg} for the CG-LOD and DG-LOD, respectively, and the known QOI of $\tilde{\varphi}_j$. We can then write the localized basis function \( \tilde{\varphi}_j^\ell \) as
\begin{equation} 
	\label{eq:locbasis}
    \tilde{\varphi}_j^\ell = \sum_{z \in \mathcal{N}_H^{\mathrm{in}}} \kappa_{zj} \Lambda_z^1 - \sum_{T \in \mathcal{T}_H} \psi_{j,T}^\ell,
\end{equation}
where the pair \( (\psi_{j,T}^\ell, \lambda_{j,T}^\ell) \in V_T^\ell \times M_T^\ell \) solves the local saddle point problem
\begin{subequations}
	\label{eq:psiTell} 
	\begin{align}
		&a (\psi_{j, T}^{\ell}, w) \;  +\hspace{-1.75ex}&  b (w,\lambda_{j, T}^{\ell}) &=  {\textstyle \sum_{z \in  \mathcal{N}_H^{\mathrm{in}}} \kappa_{zj} a_T (\Lambda_z^1, w),}&&\forall w \in V_T^\ell, \label{eq:psiTell1}\\
		&b(\psi_{j,T}^\ell,\mu) \;   &         &= {\textstyle-\frac{|T\cap\omega_j|}{|\omega_j|} \mu_j +  \sum_{z \in \mathcal{N}_H^{\mathrm{in}}} \kappa_{zj}(\mu, \Lambda_z^1)_T,}&&\forall \mu \in M_T^\ell. \label{eq:psiTell2}
	\end{align}
\end{subequations}
Note that only a small number of these problems have to be solved for each $j$. In the case of CG-LOD, the functions $\psi_{j, T}^\ell$ are non-zero only for mesh elements $T\subset\overline\omega_j$. For the DG-LOD, if the index \( j \) is associated with the characteristic function on an element $K$, then the functions \( \psi_{j, T}^\ell \) are non-zero for elements \( T \subset \mathsf{N}(K) \); for all  other indices $j$ corresponding to  basis functions of $M_H$ supported on \( K \), we have that $\tilde{\varphi}_j^\ell=-\psi_{j,K}^{\ell}$. We emphasize that, compared to naive localization strategies as in~\cite{MalP14,Maier2021}, the computational cost of the proposed stabilized localization strategy differs only for the basis functions associated with the lowest-order QOI. Specifically, these are the hat functions for the CG-LOD and the characteristic functions of elements for the DG-LOD. 

The LOD method seeks the unique function $\tilde u_H\in \tilde V_H^\ell$ such that
\begin{align}\label{eq:locmethod}
	a(\tilde u_H^\ell,\tilde v_H^\ell)=(f,\tilde v_H^\ell)_\Omega, 
    \quad \forall \tilde v_H^\ell \in \tilde V_H^\ell.
\end{align}
The following theorem provides convergence results for the LOD approximation.
\begin{theorem}[Localized  method]\label{thm:convergenceloc}
The localized multiscale method~\cref{eq:locmethod} is well-posed. Moreover, for any right-hand side \( f \in H^{s}(\Omega) \) in the case of the CG-LOD, and \( f \in H^s(\mathcal{T}_H) \) for the DG-LOD, with \( s \in \{0, 1, \dots, p+1\} \), we have the following error estimates:
	\begin{align}
		\| \nabla (u - \tilde{u}_H^{\ell})\|_{\Omega} & \lesssim  H^{1+s}|f|_{s,\TH} +  \ell^{(d-1)/ 2} \exp (- C_\mathrm{dec} \ell) \|f\|_{\Omega},  \label{eq:errestpracH1}\\
\|u - \tilde{u}_H^{\ell} \|_{\Omega} & \lesssim \big(H  +\ell^{(d-1)/ 2} \exp (- C_\mathrm{dec} \ell)\big)\| \nabla (u - \tilde{u}_H^{\ell})\|_{\Omega}.  \label{eq:errestpracL2}
	\end{align}
\end{theorem}
\begin{proof}
The proof of the above error estimates closely follows the arguments in \cite[Thm.~6.2]{Dong2023} and is therefore omitted for the sake of brevity.
\end{proof}

\section{More general problems}\label{sec:general}
This section demonstrates that the general higher-order LOD framework developed in the previous sections for the elliptic model problem can be readily extended to more complex problems. This is illustrated with the examples of the heterogeneous Helmholtz problem and the Gross–Pitaevskii eigenvalue problem.

\subsection{Heterogeneous Helmholtz problem}\label{sec:Helmho}
In comparison with the second-order diffusion-type problems introduced in \cref{sec:modelproblem}, the Helmholtz problem additionally features a zeroth-order term with a sign opposite to that of the second-order operator. The strength of this term is determined by the wavenumber \(\kappa > 0\). For large \(\kappa\), the problem becomes strongly indefinite and its solution is oscillatory, which typically results in significant numerical challenges, cf.~\cite{Babuka1997}.
The Helmholtz-type problem considered here seeks a complex-valued solution \( u : \Omega \to \mathbb{C} \) satisfying
\begin{equation}\label{eq:Helmholtz}
	\left\{
	\begin{aligned}
		- \mathrm{div}(A \nabla u) - \kappa^2 \calV^2 u &= f 
		&& \text{in } \Omega, \\
		A \nabla u \cdot \nu - i \kappa \sigma u &= 0
		&& \text{on } \partial \Omega,
	\end{aligned}
	\right.
\end{equation}
where \( A \) is a matrix-valued coefficient as in \cref{sec:modelproblem}, \( \calV \in L^\infty(\Omega; \mathbb{R}) \) satisfies the uniform bounds $0 < \calV_{\mathrm{min}} \leq \calV(x) \leq \calV_{\mathrm{max}} < \infty$ almost everywhere in $\Omega$, and \( \sigma \in L^\infty(\partial \Omega; \mathbb{R}) \) is almost everywhere positive on \( \partial \Omega \). The right-hand side~\( f \) belongs to \( L^2(\Omega; \mathbb{C}) \). This problem describes the acoustic wave propagation in heterogeneous media and is a generalization of the classical (homogeneous) Helmholtz problem. 

The weak formulation of~\eqref{eq:Helmholtz} seeks a function \( u \in V \coloneqq H^1(\Omega;\mathbb{C}) \) such that
\begin{equation}
	\label{eq:wfhelmholtz}
	a(u, v) = \int_\Omega f \bar{v} \, \mathrm{d}x, \quad \forall v \in V,
\end{equation}
where \( \bar{v} \) is the complex conjugate of \( v \), and \( a \colon V \times V \to \mathbb{C} \) is defined by
\begin{align*}
	a(u, v) 
	\coloneqq \int_\Omega A \nabla u \cdot \nabla \bar{v} \, \mathrm{d}x
	- \kappa^2 \int_\Omega \calV^2 u \bar{v} \, \mathrm{d}x
	- i \kappa \int_{\partial \Omega} \sigma u \bar{v} \, \mathrm{d}s.
\end{align*}
Note that $a$ is sesquilinear. 
A natural norm for the Helmholtz problem is
\begin{equation}
	\label{eq:hhnorm}
	\|v\|_\kappa^2 \coloneqq \|A^{1/2} \nabla v\|_{\Omega}^2 + \kappa^2 \|\calV v\|_{\Omega}^2.
\end{equation}
Henceforth, we assume that the weak formulation~\cref{eq:wfhelmholtz} of the Helmholtz problem is well-posed and admits a unique solution that satisfies the stability estimate
\begin{equation}
	\label{eq:stabhelmholtz}
	\|u\|_\kappa \lesssim \kappa^n \|f\|_{\Omega},
\end{equation}
for some $n\geq 0$. For homogeneous coefficients and general Lipschitz domains, the well-posedness of~\cref{eq:stabhelmholtz} can be proved with $n =  5/2$; see~\cite{Esterhazy2011}. In the case of variable coefficients, the analysis becomes substantially more involved; see, e.g.,~\cite{Graham2019,ChaumontFrelet2023}. Note that the polynomial-in-$\kappa$ stability assumption \cref{eq:stabhelmholtz}  is  classical  for the numerical analysis of  Helmholtz problems. 

Several works use and analyze the LOD method for Helmholtz problems, including \cite{Peterseim2016,Brown2017,Peterseim2020,MaiV22,Hauck2022}; see also \cite{Freese2024}. A key ingredient in the corresponding  analysis is the coercivity of the sesquilinear form \( a \) when restricted to the fine-scale space \( W \), defined similarly to \cref{eq:defW}. The following lemma establishes this property for both the proposed CG-LOD and DG-LOD method.

\begin{lemma}[Coercivity on fine-scale space]\label{helm:LemCoer}
Assume the resolution condition $
	{H \kappa}{p^{-1}} \leq {\sqrt{\alpha}}(\sqrt{2}\, V_{\max} C_\mathrm{ap})^{-1}
$ with a constant \( C_\mathrm{ap} > 0 \) depending  solely on the shape-regularity of \( \mathcal{T}_H \). Then the sesquilinear form \( a \) is coercive on \( W \times W \), i.e.,
\begin{equation*}
	\mathfrak{R} a(w, w) \geq \frac{\alpha}{2} \| \nabla w \|_{\Omega}^2,
	\quad \forall w \in W.
\end{equation*}
\end{lemma}
\begin{proof}
We use the uniform bounds on the coefficients \( A \) and \( \calV \), and the approximation property of the \( L^2 \)-projection \( \Pi_H \colon L^2(\Omega) \to M_H \), given by
\[
\|v - \Pi_H v\|_{\Omega} \leq  {C_\mathrm{ap}H}{p}^{-1} \|\nabla v\|_{\Omega},
\]
for a constant \( C_\mathrm{ap} > 0 \) independent of \( H \) and \( p \). Recall that we choose $M_H = \mathcal P^p(\TH)$ for the DG-LOD and $M_H = \mathcal P^p(\TH)\cap H^1(\Omega)$ for the DG-LOD. Such a $p$-explicit approximation result can be found, for instance, in \cite[Sec.~3]{Babuka1987}. 
 This approximation result then yields, for any $w \in W$, the following estimate:
\begin{align*}
	\mathfrak{R} a(w, w)
	\geq \alpha \|\nabla w\|_{\Omega}^2 - \kappa^2 \calV_{\max}^2 \|w\|_{\Omega}^2 
	\geq \big( \alpha - \kappa^2 \calV_{\max}^2 C_\mathrm{ap}^2 {H^2}{p^{-2}} \big) \|\nabla w\|_{\Omega}^2.
\end{align*}
	The desired coercivity estimate then follows with the resolution condition.
\end{proof}

 Note that, unlike in the rest of the paper, we have now explicitly tracked the $p$-dependence to examine how the polynomial degree influences the resolution condition and to enable a comparison with other methods; see \cref{rem:hpfem}.

To construct the basis functions of the LOD trial space, one solves local corrector problems analogous to~\cref{eq:psiTell}, where the sesquilinear forms \( a \) and \( a_T \) now are those of the Helmholtz problem. 
The well-posedness of these problems follows directly from classical inf--sup theory (see, e.g., \cite[Cor.~4.2.1]{BofBF13}, which extends to the complex-valued setting), noting that the sesquilinear form \( a \) is coercive on the kernel of \( b \) as shown in \cref{helm:LemCoer} and the inf--sup condition \cref{infsupbglob}.
 The global trial space of the LOD is then defined as the span of the resulting basis functions, in the spirit of~\cref{eq:approxspace,eq:locbasis}. 
Since the Helmholtz problem is non-Hermitian, the LOD method employs different trial and test spaces.
 However, owing to the specific structure of the Helmholtz equation, the test space can  be obtained as the complex conjugate of the trial space; see~\cite[Eq.~(4.6)]{Peterseim2016}. 
Therefore, no additional computations for the test space basis functions are required, and the LOD approximation to the Helmholtz problem is given by the solution \( \tilde{u}_h^\ell \in \tilde{V}_H^\ell \) satisfying
\begin{equation*}
		a(\tilde u_H^\ell,\overline{\tilde v_H^\ell})=\int_\Omega f\tilde v_H^\ell\dx,    \quad \forall \tilde v_H^\ell \in \tilde V_H^\ell.
\end{equation*}

The following theorem establishes the convergence of the CG-LOD and DG-LOD methods for the Helmholtz problem.
 
\begin{theorem}[Localized method for Helmholtz]
	\label{thm:helmh}
	Under the resolution condition from \cref{helm:LemCoer} and the oversampling condition \(\ell \gtrsim \log(\kappa)\), the LOD method for~\cref{eq:wfhelmholtz} is well-posed. 
	Moreover, with \( f \in H^{s}(\Omega,\mathbb C) \) for the CG-LOD, and \( f \in H^s(\TH; \mathbb{C}) \) for the DG-LOD, and \( s \in \{0, 1, \dots, p+1\} \), we have
	\begin{align}
		\|u - \tilde{u}_H^{\ell}\|_{\kappa} \lesssim H^{1+s} |f|_{s,\TH} + \kappa^n \ell^{\frac{d-1}{2}} \exp(-C_\mathrm{dec} \ell) \|f\|_{\Omega} \label{eq:errestpracH1helmholtz}
	\end{align}
	with  \(n\) from \cref{eq:stabhelmholtz} and the decay rate \(C_\mathrm{dec} > 0\) is independent of \(H\), \(\ell\), and \(\kappa\).
\end{theorem}
\begin{proof}
The proof combines arguments from the LOD analysis for the elliptic model problem (see, e.g., \cref{thm:dec,thm:locerr}) with techniques specific to the Helmholtz setting (see, e.g.,~\cite{Peterseim2016,Hauck2022}). For brevity, the details are omitted.
\end{proof}

\begin{remark}[Analogies to $hp$-FEM]\label{rem:hpfem}
The assumptions on the discretization parameters required for the stability and quasi-optimality of the $hp$-FEM for Helmholtz problems (see, e.g., ~\cite{Melenk2011}) are notably similar to those in the above theorem. In the case of homogeneous coefficients (for an extension to piecewise smooth coefficients, see \cite{Bernkopf2024}), it was shown that a stable and quasi-optimal finite element approximation can be achieved if the polynomial degree satisfies \( p \approx \log(\kappa)\) and the resolution condition \( H \kappa  p^{-1} \lesssim 1 \) holds, along with a suitable refinement strategy near geometric singularities. This resolution condition is the same as in ~\cref{helm:LemCoer}  (up to a constant), and the role of the polynomial degree in the $hp$-FEM is similar to that of the localization parameter in the LOD; see~\cref{thm:helmh}.
\end{remark}

\subsection{Gross--Pitaevskii problem}
As another example, we consider the Gross--Pitaevskii problem, which arises in quantum physics as a model for quantum states of so-called Bose--Einstein condensates. The problem is posed on a convex Lipschitz domain~$\Omega$, with homogeneous boundary conditions imposed on its boundary. This corresponds to the choice $V \coloneqq H^1_0(\Omega)$. Note that due to the rapid decay of low-energy quantum states, the restriction to a sufficiently large domain, along with homogeneous Dirichlet boundary conditions, is a physically reasonable modeling assumption.
Stationary quantum states correspond to the critical points of the Gross--Pitaevskii energy functional, defined as
\begin{equation}
	\label{eq:energy}
	\mathcal{E}(v) \coloneqq \tfrac{1}{2} (\nabla v, \nabla v)_{\Omega} + \tfrac{1}{2} (\calV v, v)_{\Omega} + \tfrac{\kappa}{4} (|v|^2 v, v)_{\Omega}, \quad v \in V,
\end{equation}
subject to the \( L^2 \)-normalization constraint \( \|v\|_{\Omega} = 1 \). Here,  $\mathcal{V} \in L^\infty(\Omega)$ is a non-negative (possibly rough) trapping potential and the parameter $\kappa > 0$ characterizes the strength of repulsive interactions between particles.
Of particular physical interest is the ground state, which corresponds to the stationary quantum state of lowest energy, i.e., it solves the constrained minimization problem
\begin{equation}
	\label{eq:gs}
	u \in \argmin_{v \in V \with \|v\|_{\Omega} = 1} \mathcal{E}(v),
\end{equation}
and the minimal energy is denoted by $E \coloneqq \mathcal E(u)$.
The corresponding Euler--Lagrange equations imply that the ground state \( u \), together with an eigenvalue \( \lambda \in \mathbb{R} \), solves the nonlinear eigenvalue problem
\begin{equation}
	\label{eq:GPEweak}
	(\nabla u, \nabla v)_{\Omega} + (\calV u, v)_{\Omega} + \kappa (|u|^2 u, v)_{\Omega} = \lambda (u, v)_{\Omega} , \quad \forall v \in V,
\end{equation}
where $\lambda$ is referred to as the ground state eigenvalue.
In the above setting, it is a classical result that the ground state eigenvalue is the smallest eigenvalue among all eigenpairs of~\cref{eq:GPEweak} and it is simple. Moreover, the ground state is unique up to sign and can be chosen to be strictly positive in the interior of \( \Omega \), cf.~\cite{CanCM10}.

For the construction of the LOD method (we restrict ourselves to the prototypical method for simplicity; a localization can be performed analogously to \cref{sec:decay,sec:localization}), we consider only the terms on the left-hand side of~\cref{eq:GPEweak} which are linear in~$u$. The resulting bilinear form \( a \colon V \times V \to \mathbb{R} \) is defined as
\begin{equation}\label{eq:moda}
    a(w,v) \coloneqq (\nabla w, \nabla  v)_\Omega + (\calV w, v)_\Omega.
\end{equation}
The approximation space \( \tilde V_H \) of the prototypical LOD is then defined as in~\cref{eq:Zms}, using the modified bilinear form from~\eqref{eq:moda}. Depending on whether the space~\( M_H \), used in the definition of the fine-scale space \( W \) in \cref{eq:defW}, consists of globally continuous or $\TH$-piecewise polynomials, one obtains a CG- or DG-version of the LOD method, respectively.
The prototypical LOD approximation  is then defined as the solution to the finite-dimensional constrained minimization problem.
\begin{equation}
	\label{eq:gs_lod}
	\tilde u_H \in \argmin_{\tilde v_H \in \tilde V_H \with \|\tilde v_H\|_{\Omega} = 1} \mathcal{E}(\tilde v_H),
\end{equation}
and the corresponding minimal energy is denoted by $E_H \coloneqq \mathcal E(\tilde u_H)$. While the existence of such a minimizer follows from classical compactness arguments in the finite-dimensional setting, its uniqueness is generally not guaranteed. An overview of algorithms to practically solve  \cref{eq:gs_lod} (after a localization of the basis functions) can be found in \cite{HenJ25}.
The Euler--Lagrange equations corresponding to~\cref{eq:gs_lod} give rise to the following finite-dimensional nonlinear eigenvalue problem: seek \( \tilde u_H \in \tilde V_H \) and an associated eigenvalue \( \lambda_H \in \mathbb{R} \) such that
\begin{equation*}\label{eq:GPEweakdisc}
    (\nabla \tilde u_H, \nabla \tilde v_H)_\Omega + (\calV \tilde u_H, \tilde v_H)_\Omega + \kappa (|\tilde u_H|^2\tilde u_H,\tilde v_H)_\Omega =  \lambda_H(\tilde u_H,\tilde v_H)_\Omega ,\quad \forall \tilde v_H \in \tilde{V}_H.
\end{equation*}

As a preliminary step towards proving the convergence of the prototypical LOD method for the Gross--Pitaevskii problem, we first analyze the approximation properties of the space~\(\tilde V_H\), as done in the following lemma.

\begin{lemma}[Approximation properties of $\tilde V_H$]\label{lem:approxGPE}
    Assume that $u \in H^s(\Omega)$ as well as $|u|^2u \in H^s(\Omega)$ with \( s \in \{0, 1, \dots, p+1\} \), and let $\calP u \in \tilde V_H$ be the solution to
    \begin{equation}\label{eq:projtVH}
        a(\calP u, \tilde v_H) = - \kappa (|u|^2u,\tilde v_H)_\Omega - \lambda(u, \tilde v_H)_\Omega,\quad \forall \tilde v_H \in \tilde V_H.
    \end{equation}
    Then, we have 
    \begin{equation}
    	\label{eq:approxpropgpe}
        \| \nabla( u - \calP u) \|_\Omega \lesssim  H^{1+s} \big(| |u|^2u |_{s,\Omega} + |u|_{s,\Omega}\big).
    \end{equation}
\end{lemma}

\begin{proof}
By coercivity and symmetry of \(a\), and using~\eqref{eq:GPEweak} and~\eqref{eq:projtVH}, we obtain that
    \begin{align*}
        \|\nabla (u - \calP u)\|_\Omega^2 &\leq 
        a(u - \calP u, u) = - \kappa (|u|^2u,u - \calP u)_\Omega - \lambda(u, u - \calP u)_\Omega.
    \end{align*}
Since \( u - \calP u \in W \) by construction, and using the approximation properties of the \( L^2 \)-projection \( \Pi_H \colon L^2(\Omega) \to M_H \), it follows that
    \begin{align*}
        \|\nabla (u - \calP u)\|_\Omega^2 &\leq - \kappa (|u|^2u - \Pi_H(|u|^2u),(u - \calP u) - \Pi_H(u - \calP u))_\Omega \\&\qquad- \lambda(u - \Pi_Hu, (u - \calP u) - \Pi_H(u - \calP u))_\Omega \\
        & \lesssim H^{1+s} \big(| |u|^2u |_{s,\Omega} + |u|_{s,\Omega}\big) \|\nabla(u - \calP u)\|_\Omega,
    \end{align*}
    where the hidden constant depends on \( \lambda \). This yields \cref{eq:approxpropgpe}. 
\end{proof}

The following theorem proves the convergence of the prototypical LOD method.
\begin{theorem}[Prototypical method for the Gross--Pitaevskii problem]\label{thm:errgpe}
	
Assume that $u \in H^s(\Omega)$ and $|u|^2u \in H^s(\Omega)$ with \( s \in \{0, 1, \dots, p+1\} \). Then the discrete ground state and the corresponding energy fulfill
\begin{equation*}
	\|\nabla (u-\tilde u_H)\|_\Omega \lesssim  H^{1+s} \big(| |u|^2u |_{s,\Omega} + |u|_{s,\Omega}\big),\qquad |E-E_H| \lesssim \|\nabla (u-\tilde u_H)\|_\Omega^2.
\end{equation*}
\end{theorem}
\begin{proof}
From the general convergence theory developed in~\cite[Thm.~1]{CanCM10}, we derive under the stated assumptions that the solution to~\eqref{eq:GPEweakdisc} fulfills a quasi-best approximation property in the $H^1(\Omega)$-norm and the stated energy error holds. The result then follows directly invoking \cref{lem:approxGPE}. 
\end{proof}

\begin{remark}[$L^2$- and eigenvalue approximation]
Note that $L^2$-error estimates for the ground state and the eigenvalue can be derived as well. Under suitable regularity assumptions on the dual problem defined in the proof of~\cite[Thm.~3]{CanCM10}, we obtain an additional order of convergence in the \(L^2(\Omega)\)-norm for the ground state approximation, compared to the \(H^1\)-estimate in~\cref{thm:errgpe}.
Under similar assumptions, the eigenvalue approximation exhibits the same convergence rate as the energy approximation.
A proof of these results is beyond the scope of this work; for details, see~\cite{HenP23}, which provides the proof for the CG-LOD with \( p = 1 \).

\end{remark}
\begin{remark}[Regularity of the solution]
	Under the assumptions made in this section, one can show that \( u \) and \( |u|^2u \) belong to \( H^2(\Omega) \), with corresponding norms bounded independently of the oscillations of \( \mathcal{V} \); see~\cite[Lem.~2.2]{HenP23}. Using similar arguments and assuming more regular $\calV$ and a smooth boundary, even higher regularity of the solution can be shown, e.g., \( u \in H^{p+1}(\Omega) \), \( |u|^2u \in H^{p+1}(\Omega) \) for some $p>1$. 	These considerations justify the assumptions in \cref{lem:approxGPE,thm:errgpe}. Note that, however, the seminorms of order greater than two may no longer be bounded independently of the oscillations of \( \mathcal{V} \). 
	The additional boundary regularity assumptions are not required if \( u \) is compactly supported in~\( \Omega \). While exact compact support may not occur in practice, the ground state typically exhibits a rapid decay (see, e.g.,~\cite[Thm.~2.5]{BaoC13}), which allows to relax the boundary regularity assumptions.
	Overall, under the same regularity assumptions, the LOD solution is expected to be two orders more accurate than a classical higher-order finite element method; see also the numerical investigation in~\cite{Doe25}.
\end{remark}

\section{Numerical experiments}\label{sec:numerics}

In all numerical experiments, we consider the domain \( \Omega = (0, 1)^2 \), unless stated otherwise. For implementation purposes, we use uniform Cartesian meshes \( \mathcal{T}_H \) composed of square elements of side length \( H \), and choose \( M_H \coloneqq \mathcal{Q}^p(\mathcal{T}_H) \) for the DG-LOD and \( M_H \coloneqq \mathcal{Q}^p(\mathcal{T}_H) \cap H^1(\Omega) \) for the CG-LOD, where \( \mathcal{Q}^p(\mathcal{T}_H) \) denotes the space of $\TH$-piecewise polynomials of coordinate degree at most \( p \); see also the footnote on \cpageref{page:footnote}. The local (infinite-dimensional) patch problems~\cref{eq:defKTell} are discretized using local submeshes of a fine Cartesian mesh \( \mathcal{T}_h \), with mesh size \( h < H \), fine enough to resolve all microscopic features of the coefficients. For this fine-scale discretization, we use the \( \mathcal{Q}^q \)-finite element method, with the polynomial degree \( q \in \mathbb{N} \) specified individually in the subsections below.
In the fully discrete convergence analysis, the space \( V \) is replaced by the fine-scale finite element space, and most arguments carry over directly; see, e.g., \cite[Ch.~4.4]{MalP20}. This yields an a priori error estimate for the fully discrete LOD approximation with respect to the fine-scale finite element solution, analogous to \cref{thm:convergenceloc}. An estimate with respect to the weak solution of the original PDE then follows by applying the triangle inequality and standard finite element approximation results.

The numerical experiments presented below can be reproduced using the code available at \url{https://github.com/moimmahauck/HO_LOD}.

\subsection{Heterogeneous elliptic problem}
\begin{figure}
	\begin{minipage}[c]{.32\linewidth}
\includegraphics[width=\linewidth]{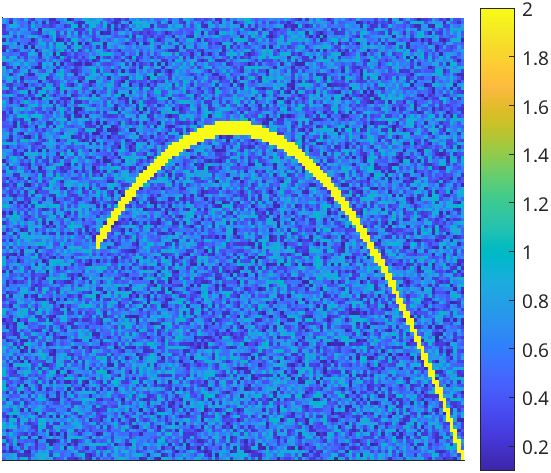}
	\end{minipage}
	\hfill
	\begin{minipage}[c]{.32\linewidth}
		\includegraphics[width=\linewidth]{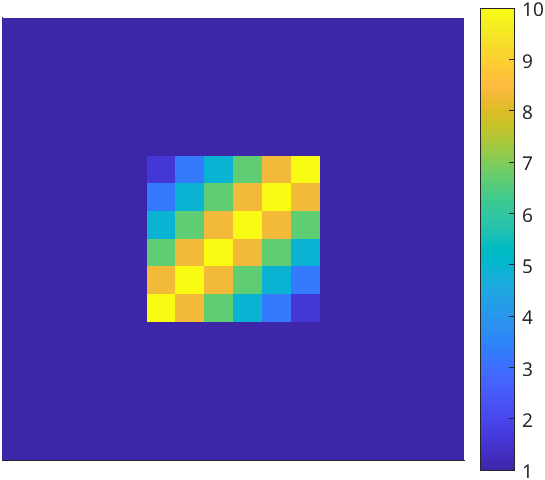}
	\end{minipage}
	\hfill
	\begin{minipage}[c]{.32\linewidth}
		\includegraphics[width=\linewidth]{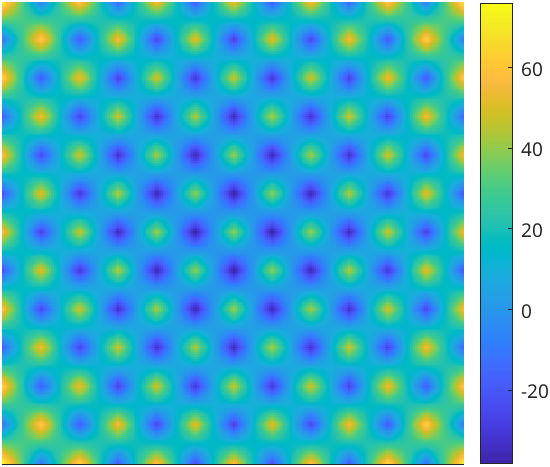}
	\end{minipage}		
	\caption{Coefficients $A_1$ and $A_2$ used in the first and second numerical experiments, respectively (left and center), and potential~$\calV$ used in the third numerical experiment (right).}
	\label{fig:ellipticcoefficient}
\end{figure}
As a first numerical experiment, we consider the elliptic model problem \cref{eq:weakform} with the coefficient $A_1$ shown in \cref{fig:ellipticcoefficient} (left). The coefficient is piecewise constant on a uniform Cartesian grid $\mathcal{T}_\epsilon$ with mesh size $\epsilon = 2^{-7}$. For all elements whose midpoints lie within a distance of $4\epsilon$ from a given parabola, the coefficient is set to a value of 2. In all other elements, the coefficient values are sampled independently from a uniform distribution on the interval $[0.1, 1]$.
We further consider the two smooth source terms
\begin{align*} 
f_1(x,y) = 2\pi^2\sin(\pi x)\sin(\pi y),\qquad f_2(x,y) =1. 
\end{align*} 
We emphasize that for the  source term $f_2$, as observed in \cref{thm:convergenceloc}, the first term on the right-hand side of error estimate \cref{eq:errestpracH1} vanishes. This implies that, for~$f_2$, only the second summand, i.e., the exponentially decaying localization error, remains.
We abbreviate the relative errors with respect to the energy norm as
\begin{equation*}
\mathrm{err}_a(H,\ell) \coloneqq \frac{\|u_h - u_{H,h}^\ell\|_a}{\|u_h\|_a},
\end{equation*}
where $u_{H,h}^\ell$ denotes the fine-scale LOD approximation, and $u_h$ is the fine-scale finite element solution as reference. The fine-scale discretization is performed using the $\mathcal{Q}^1$-finite element method on the fine mesh $\mathcal{T}_h$ with mesh size $h = 2^{-9}$, which is sufficiently fine to resolve the microscopic details of the coefficient $A$. Note that, due to the expected low regularity of the analytical solution, higher-order finite elements do not offer an advantage in terms of convergence rates at the fine scale.

\begin{figure}
	\includegraphics[height=.32\linewidth]{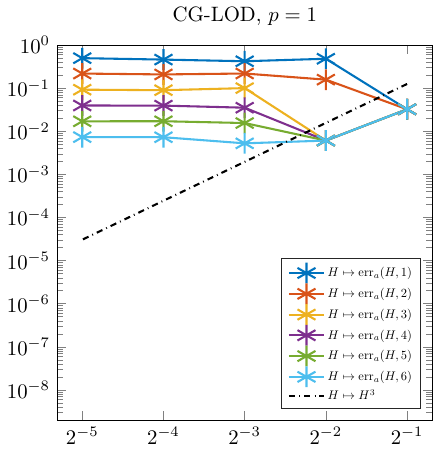}\hfill
	\includegraphics[height=.32\linewidth]{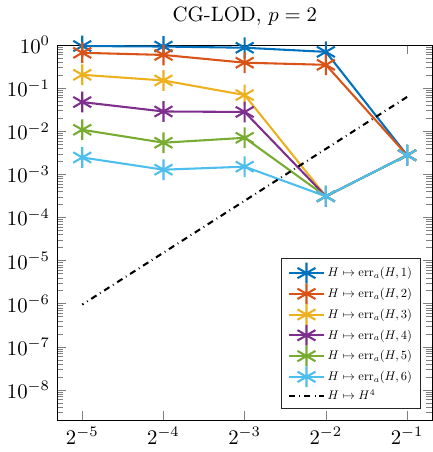}\hfill
	\includegraphics[height=.32\linewidth]{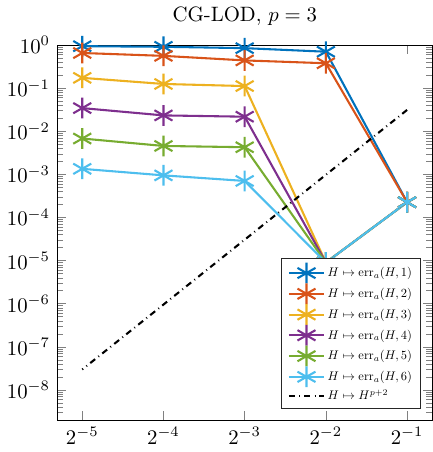}\\[2ex]
	\includegraphics[width=.32\linewidth]{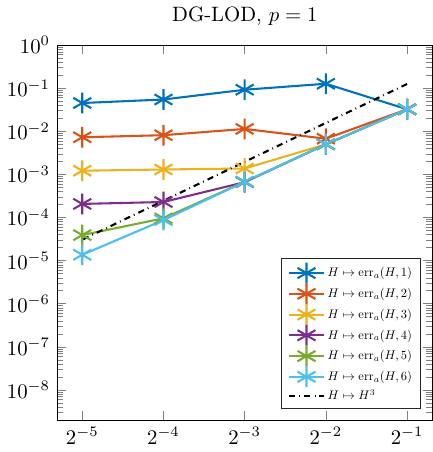}\hfill
	\includegraphics[width=.32\linewidth]{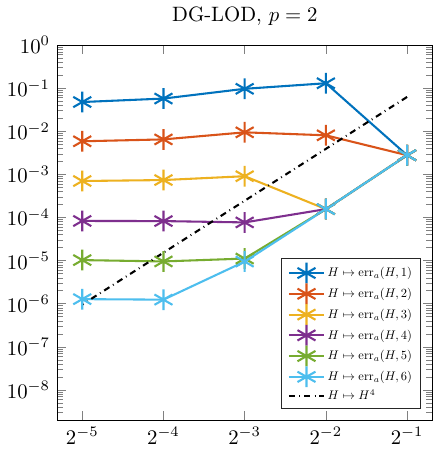}\hfill
	\includegraphics[width=.32\linewidth]{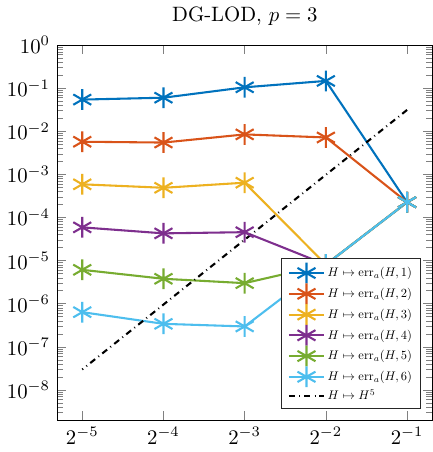}
	\caption{Error plots for the CG-LOD and DG-LOD methods for polynomial degrees $p \in \{1,2,3\}$ for the source term $f_1$. For fixed oversampling parameters $\ell$, the relative energy norm error is plotted as a function of coarse mesh size~$H$.}
	\label{fig:conv}
\end{figure}

In \cref{fig:conv}, the first row of plots illustrates the convergence behavior of the CG-LOD method. One can observe that the expected convergence rates can still be inferred from the first two data points, where all patches are global and the localization error is effectively zero. However, for mesh sizes $H$ beyond $2^{-2}$, the localization error completely dominates the convergence behavior.
Although this observation is consistent with the theoretical results in \cref{thm:convergenceloc}, this numerical experiment clearly shows the unsatisfactory localization properties of the CG-LOD.
The situation is quite different for the DG-LOD. The second row of plots in \cref{fig:conv} clearly shows a convergence of order $p+2$ for the LOD approximation as predicted by \cref{thm:convergenceloc}, provided that the oversampling parameter is sufficiently large. The noticeably larger gaps between the plateaus of the error curves for different oversampling parameters showcase the significantly improved localization properties of the DG-LOD compared to the CG-LOD. This improvement can be attributed to the use of the discontinuous Galerkin ansatz for the QOI.

\begin{figure}
	\includegraphics[height=.32\linewidth]{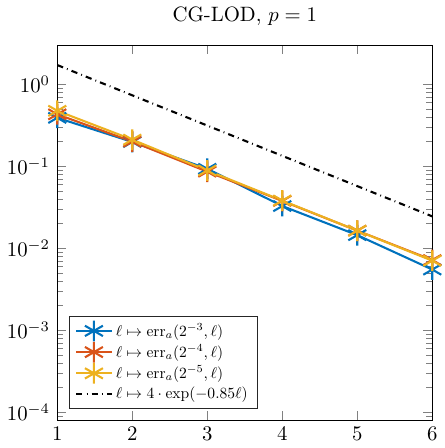}\hfill
	\includegraphics[height=.32\linewidth]{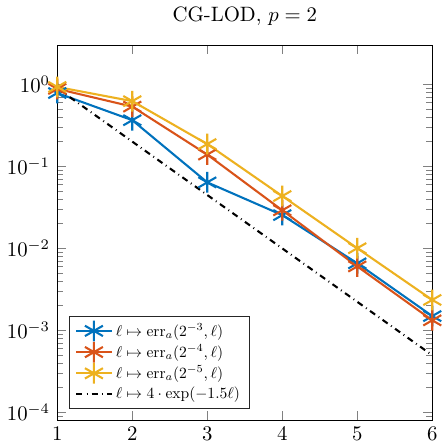}\hfill
	\includegraphics[height=.32\linewidth]{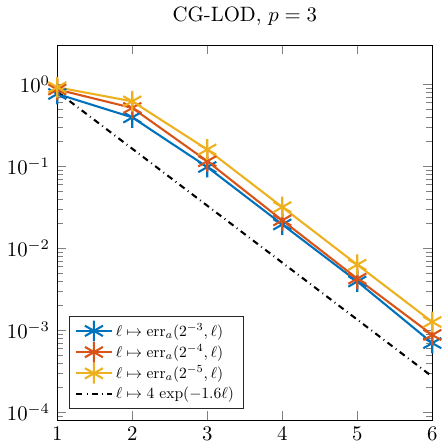}\\[2ex]
	\includegraphics[width=.32\linewidth]{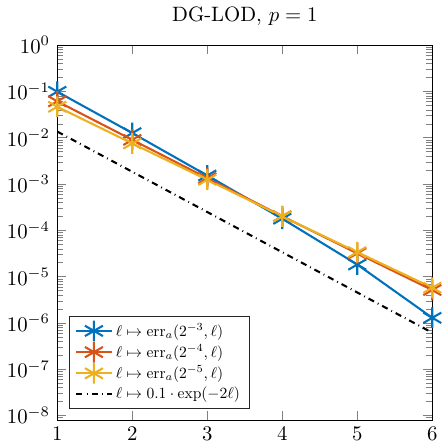}\hfill
	\includegraphics[width=.32\linewidth]{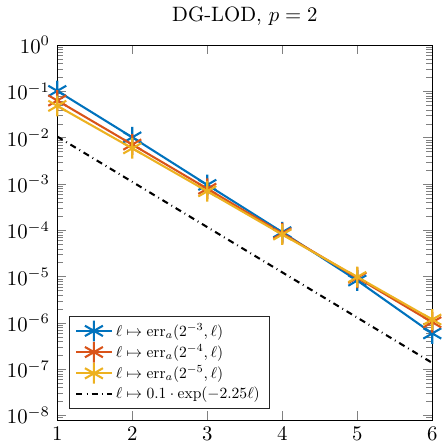}\hfill
	\includegraphics[width=.32\linewidth]{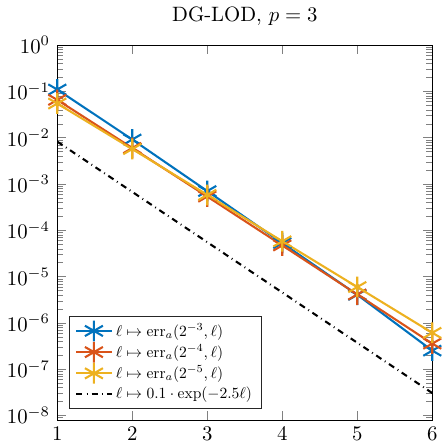}
	\caption{Error plots for the CG-LOD and DG-LOD methods for polynomial degrees $p \in \{1,2,3\}$ for the source term $f_2$. For fixed coarse mesh sizes $H$, the relative energy norm error is plotted as a function of the oversampling parameter $\ell$.}
	\label{fig:loc}
\end{figure}

In \cref{fig:loc} we compare the exponential decay of the localization error for the CG-LOD (first row) and the DG-LOD (second row). Recall that for the source term $f_2$ used in this experiment, the first term in the error estimate \cref{eq:errestpracH1} vanishes, leaving only the exponentially decaying localization error.
The numerical results in \cref{fig:loc} confirm the completely different localization behavior of the two methods. Note that, in general, increasing the polynomial degree leads to better localization, i.e., faster decay of the localization error. 

\subsection{High frequency heterogeneous Helmholtz problem}
\begin{figure}
	\begin{minipage}[b]{.32\linewidth}
	\includegraphics[width=\linewidth]{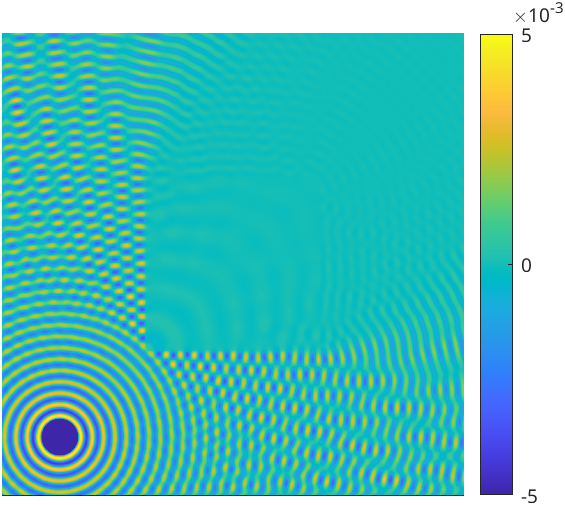}
	\end{minipage}
	\hspace{.5cm}
	\begin{minipage}[b]{.32\linewidth}
	\includegraphics[width=\linewidth]{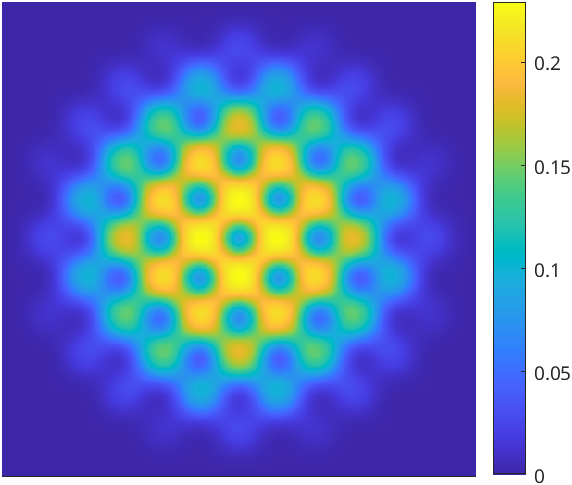}
	\end{minipage}
	\caption{Real part of the Helmholtz LOD solution (left), and LOD Gross--Pitaevskii ground state approximation (right)}
	\label{fig:helmholtz}
\end{figure}

In the second numerical experiment, we demonstrate the performance of the higher-order DG-LOD in the context of high-frequency heterogeneous Helmholtz problems. The coefficient~$A_2$ used in this experiment is shown in \cref{fig:helmholtz} (center), and the wave number is chosen as $\kappa = 2^8$. Furthermore, the coefficients $\calV$ and $\sigma$ are set to one.
As the source term $f_3$, we use an approximate point source located at~$(1/8, 1/8)$, which is supported within a circle of radius $1/20$, defined as
\begin{align*}
	f_3(x,y) = \begin{cases}
		10^4 \cdot  \exp\bigg(\frac{-1}{1-\tfrac{(x - 1/8)^2 + (y - 1/8)^2}{1/20^2}}\bigg)& \text{if }(x - 1/8)^2 + (y - 1/8)^2 < 1/20^2,\\
		0& \text{else.}
	\end{cases}
\end{align*} 
A high wavenumber, such as $\kappa = 2^8$, leads to a pronounced pollution effect when classical low-order finite elements are used; see \cite{Babuka1997}. As a result, obtaining a quasi-optimal approximation demands considerably more restrictive conditions on the fine-scale mesh size than are required merely to capture the oscillatory nature of the solution.
In this numerical experiment, we use the $\mathcal Q^2$-finite element method on the mesh $\Th$ with $h = 2^{-10}$ for the fine-scale discretization. This choice aligns with the mesh size condition $h \sim \kappa^{-5/4}$ for quadratic finite elements to achieve a quasi-optimal approximation; see \cite[Thm.~5.1]{Du2015} for the constant coefficient~case.

\cref{fig:helmholtz} (left) shows the real part of the LOD approximation for $p = 3$, $H = 2^{-6}$, and $\ell = 4$. The corresponding relative error in the norm $\|\cdot\|_\kappa$, defined in~\cref{eq:hhnorm}, against the fine-scale solution is $7.2970 \times 10^{-4}$. This result demonstrates that higher polynomial degrees on the coarse scale can further relax the resolution requirements, which is consistent with the theoretical result presented in \cref{helm:LemCoer}.

\subsection{Gross--Pitaevskii eigenvalue problem}

In the third numerical experiment, we apply the higher-order DG-LOD method to approximate the Gross--Pitaevskii ground state. We consider the domain \(\Omega = (-6,6)^2\), the particle interaction parameter \(\kappa = 100\), and choose the potential defined as
\begin{equation*}
\mathcal V(x) \coloneqq \tfrac12|x|^2 + 40\times \mathrm{tent}(x)\mathrm{tent}(y),
\end{equation*}
where \( \mathrm{tent} \) denotes the periodized version of the tent function on the interval \( [0,1] \), which attains the value one at $0.5$ and vanishes at $0$ and $1$. This potential is illustrated in \cref{fig:ellipticcoefficient} (right) and a corresponding ground state approximation is shown n \cref{fig:helmholtz} (right). Since \( \calV \in W^{1,\infty}(\Omega) \), the DG-LOD method with degree \( p = 2 \) achieves optimal order convergence. This is because the assumption in \cref{thm:errgpe}, that \( u \) and \( |u|^2 u \) belong to \( H^3(\Omega) \), is satisfied.

In the convergence plot shown in~\cref{fig:gpe2}, one observes the expected optimal rates of convergence,  for the ground state approximation in the \( L^2 \)- and \( H^1 \)-norms, as well as for the energy and eigenvalue approximations, provided the oversampling parameter is chosen sufficiently large. We emphasize that, initially, for very coarse mesh sizes \( H \), a reduced convergence order (e.g.,  around three for the $H^1$-error of the ground state approximation) is observed.  This behavior is based on the fact that seminorms of order greater than two for \( u \) and \( |u|^2u \) are not bounded independently of the oscillations in \( \calV \). As soon as these oscillations are sufficiently resolved, the theoretically predicted optimal convergence rates are attained. A similar behavior is observed for the \( L^2 \)-error, as well as for the energy and eigenvalue errors. Note that machine precision effects can be seen in the eigenvalue errors for small mesh~sizes.

\begin{figure}
	\includegraphics[width=.32\linewidth]{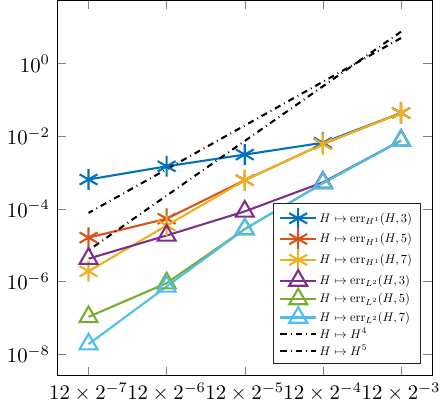}\hspace{.5cm}
	\includegraphics[width=.32\linewidth]{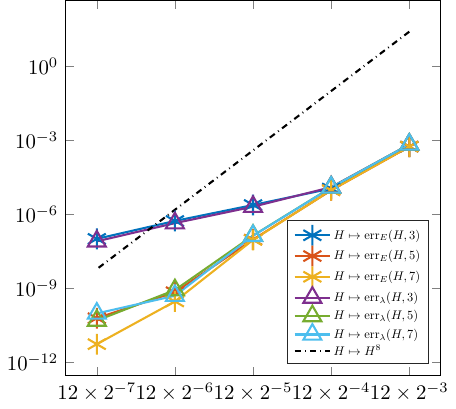}
	\caption{Error plots for the DG-LOD with $p = 2$ for the Gross-Pitaevskii problem. For fixed oversampling parameters \( \ell \), the \( L^2 \)- and \( H^1 \)-errors of the ground state approximation (left), as well as the errors in the energy and eigenvalue approximations (right), are plotted as functions of the coarse mesh size~\( H \).}
	\label{fig:gpe2}
\end{figure}

\section*{Acknowledgments}
M.~Hauck and R.~Maier acknowledge funding from the Deutsche Forschungsgemeinschaft (DFG, German Research Foundation) -- Project-ID 258734477 -- SFB 1173. Part of this research was performed while the authors were visiting the Institute for Mathematical and Statistical Innovation (IMSI), which is supported by the National Science Foundation (Grant No. DMS-1929348).

\bibliographystyle{alpha}
\bibliography{bib}
\end{document}